\documentclass[lettersize,journal]{IEEEtran}

\usepackage{mathtools}
\usepackage{amsmath,amsfonts}
\usepackage{algorithmic}
\usepackage{algorithm}
\usepackage{array}
\usepackage[caption=false,font=normalsize,labelfont=sf,textfont=sf]{subfig}
\usepackage{textcomp}
\usepackage{stfloats}
\usepackage{url}
\usepackage{verbatim}
\usepackage{graphicx}
\usepackage{cite}
\usepackage[latin1]{inputenc}
\usepackage{amssymb,amsmath}
\usepackage{graphics}
\usepackage{float}
\usepackage{color}
\usepackage{dsfont}
\usepackage{bm}
\usepackage{upgreek}
\usepackage{arydshln}
\usepackage{enumerate}
\usepackage{amsthm}
\usepackage{graphicx}
\usepackage{threeparttable}
\usepackage{caption}
\usepackage{multirow}
\usepackage{color}
\usepackage{xfrac}
\usepackage{array}
\usepackage{nomencl}
\usepackage{hyphenat}
\usepackage{algorithm,algorithmic}
\usepackage{url}
\usepackage[font=small]{caption}
\usepackage{hyphenat}
\usepackage{setspace}
\usepackage{cancel}
\usepackage[dvipsnames]{xcolor}
\usepackage{epsfig}   

\usepackage{xr}
\usepackage{pdfpages}

\def\QED{\mbox{\rule[0pt]{1.4ex}{1.4ex}}} % ajout de JP
\usepackage{soul}  % ajout de JP

\newcommand\tra{{\rm Tr}}

\newcommand\cov{{\rm Cov}}
\newcommand\var{{\rm Var}}
\newcommand\ve{{\rm vec}}

\newcommand\e{{\rm E}}
\newcommand\pardef{ \stackrel{{\rm def}}{=} }

\usepackage{chngcntr}
\usepackage{apptools}
\usepackage{enumitem}

\newtheorem{lemma}{Lemma}
\newtheorem{result}{Result}
\newtheorem{definition}{Definition}

\begin{document}

\title{On the parametric and semiparametric Fisher information matrix for non-zero mean stationary spherical  invariant random processes}
\author{Jean-Pierre Delmas, Habti Abeida, and Stefano Fortunati
\thanks{
Jean-Pierre Delmas and Stefano Fortunati are with Samovar laboratory, Telecom SudParis, Institut Polytechnique de Paris, 91120 Palaiseau, France,
e-mail: jean-pierre.delmas@it-sudparis.eu, stefano.fortunati@telecom-sudparis.eu.
Habti Abeida is with Taif University, College of Engineering, Dept. of  Electrical Engineering, Al-Haweiah, 21974, Saudi Arabia,
email: h.abeida@tu.edu.sa,
}}
\maketitle

%%%%%%%%%%%%%%%%%%%%%%%%%%%%%%%%%%%%%%%%%%%%%%%%ù
%\setcounter{equation}{102}
%%%%%%%%%%%%%%%%%%%%%%%%%%%%%%%%%%%%%
%%%%%%%%%%%%%%%%%%%%%%%%%%%%%%%%%%%%%%%%%%%%%%%%%%%%%%%%%%%%%%%%%%%%
%%%%%%%%%%%%%%%%%%%%%%%%%%%%%%%%%%%%%%%%%%%%%%%%%%%%%%%%%%%%%%%%%\`{u}\`{u}
%
\begin{abstract}
The classical Whittle formula provides a closed-form expression for the asymptotic Fisher information matrix (FIM) rate of multidimensional, real-valued, zero-mean, purely nondeterministic stationary Gaussian processes (GPs), expressed in terms of their parameterized spectra within a maximum-likelihood framework. However, the Gaussian assumption underlying this result restricts its applicability to many real-world signals exhibiting heavy-tailed or non-Gaussian behavior.
In this paper, we extend Whittle's result to multidimensional, real-valued stationary compound Gaussian processes (CGPs) with arbitrary mean, under a unified framework encompassing fully known, parameterized, and completely unknown density generators. Building upon the Slepian-Bangs formula for $n$ consecutive observations and extending it to the semiparametric setting, we leverage the asymptotic properties of block Toeplitz matrices to obtain a closed-form spectral-domain expression for the asymptotic FIM rate.
The resulting expression, common to all density generator families, generalizes the standard zero-mean Gaussian formula by incorporating a distribution-dependent contribution associated with the nonzero mean, and an additional covariance term that is invariant to the choice of non-Gaussian distribution. This extension enables efficient and theoretically grounded performance analysis for heavy-tailed signal processing applications.
\end{abstract}
%%%%%%%%%%%%%%%%%%%%%%%%%%%%%%%%%%%%%%%%%%%%%%%%%%%%%%%%%%%%%%%%%%%%%%%%%%%%%%


\centerline{
Submitted to IEEE Trans. SP
}

\date{\today}

\section{Introduction}
\label{sec:Introduction}
%%%%%%%%%%%%%%%%%%%%%%%%%%%%%%%%%%%%%%%%%%%%%%%%%%
%
Gaussian processes (GPs) are often inadequate for capturing the complexities of real-world data, particularly when it exhibits heavy-tailed characteristics.
For instance, they frequently fail to accurately characterize radar clutter, noise and interference in indoor and outdoor mobile communication channels, as well as noise encountered in imaging applications. 
To overcome these shortcomings, compound Gaussian processes (CGPs), referred to in the engineering literature as spherically invariant random processes (SIRPs) \cite{Yao2003}, provide a flexible and effective modeling framework for such phenomena and have consequently found numerous engineering applications. 
This class of processes encompasses a broad range of stochastic models, including, among others, Student-$t$ processes, $K$-distributed processes, 
$\epsilon$-contaminated Gaussian processes, as well as specific instances of symmetric $\alpha$-stable processes, thereby providing a unified framework for robust stochastic modeling under non-Gaussian conditions.
Beyond the independent and identically distributed setting, one of the most fundamental and practically relevant extensions is the second-order (wide-sense) stationary parametric model. In this framework, the dependence structure is characterized by a finite-dimensional set of parameters governing the first- and second-order statistics, typically the mean and/or the spectrum, which must be estimated from $n$ consecutive observations. 
%This class encompasses widely used models such as autoregressive (AR), autoregressive moving-average (ARMA), and other linear stochastic processes, which provide a favorable trade-off between analytical tractability and the ability to capture the correlation structure of real-world signals. Stationarity further imposes time-invariance of the second-order statistics, thereby enabling efficient statistical inference, spectral analysis, and asymptotic characterization.

Evaluating the accuracy of parametric estimators within a given probabilistic framework is fundamental to the analysis and design of signal processing systems.
In this context, the Cramér-Rao bound (CRB), defined as the inverse of the FIM, 
generally provides a lower bound on the variance of unbiased estimators and therefore constitutes a standard benchmark for assessing estimation performance.
For Gaussian models, closed-form expressions of the FIM based on $n$ consecutive observations, commonly known as the Slepian-Bangs formulas, were established in \cite{Slepian1954,Bangs1971}, \cite{Stoica1997}, and \cite{Delmas2004} for real-valued, circular complex, and noncircular complex Gaussian random variables, respectively. More recently, these results have been generalized to the broader class of circular and noncircular complex elliptically symmetric (ES) distributions. Specifically,
 extensions were established in \cite{Besson2013,Greco2013,Abeida2019} under the assumption of a known density generators, in \cite{Fortunati2019,Fortunati2026} when the density generator is treated as an infinite-dimensional nuisance parameter, leading to semiparametric FIM, and in \cite{Abeida2023} when it is parameterized by a finite-dimensional nuisance parameter. These advances enable more realistic performance bounds for modern signal processing applications where data often exhibit non-Gaussian, heavy-tailed behavior.

However, the direct evaluation of the Slepian-Bangs formulas requires inversions of $mn\times mn$ covariance matrices, resulting in a computational complexity that rapidly becomes prohibitive for high-dimensional observations $m$, typically on the order of $m^3n^3$. To circumvent this limitation, Whittle \cite{Whittle1953} introduced  an asymptotic approximation of the FIM, now commonly referred to as Whittle's formula, for zero mean, purely nondeterministic stationary GPs.
Beyond its substantially lower computational burden, Whittle's approximation admits an elegant spectral-domain interpretation of the information content carried by the process, thereby establishing a direct connection between statistical efficiency and spectral characteristics. Owing to these properties, it has become a central analytical tool in the study of Gaussian time-series models, including ARMA  \cite[p. 240-242]{Box1970}. Whittle originally derived this result from the asymptotic covariance analysis of least-squares estimators in the multivariate setting, while the extension to non-zero mean processes was subsequently addressed in \cite{Zeira1990}.
Nevertheless, it was later shown in \cite{Porat1995} that the relative error of the proposed approximation does not necessarily vanish asymptotically as the number of observations tends to infinity, thereby raising important questions regarding its general validity. Subsequently, Porat established in \cite[Th. 5.3]{Porat1993} the asymptotic validity of Whittle's formula in the univariate case using linear prediction arguments. Importantly, these derivations fundamentally rely on the closure of Gaussian distributions under summation of independent random variables, a structural property that generally fails to hold for independent CG random variables. Consequently, extending Whittle-type asymptotic FIM approximations to non-Gaussian CGPs  remains a nontrivial and theoretically challenging problem.

Such an extension can instead be achieved by exploiting the Toeplitz structure of the covariance matrix associated with $n$ consecutive observations of a stationary CGP and by leveraging asymptotic results on Toeplitz matrices based on the concept of asymptotically equivalent sequences introduced by Gray \cite{Gray2006}, which provides an accessible framework for Szegö's theory \cite{Grenander1958}, which plays a central role in the spectral analysis of large structured covariance matrices.
Preliminary progress in this direction was recently reported in \cite{Delmas2024} albeit limited to the special case of zero mean, univariate, real-valued CGPs. Moreover, the formula derived therein relies on the results of \cite{Radaelli2023}, which are limited to stationary processes with finite Markov order, and involves coefficients whose analytical expressions are not explicitly available.
In contrast, building upon asymptotic results for block Toeplitz matrices \cite{Gutierrez2008,Gutierrez2011,Gutierrez2019}, the present paper derives a closed-form expression for the asymptotic FIM rate. This result can be interpreted as a genuine multidimensional extension of Whittle's formula to real-valued, stationary CGPs, without requiring finite Markov order and allowing for an arbitrary (possibly non-zero) mean.
Furthermore, the proposed framework unifies three important statistical scenarios: (i) fully known density generators, (ii) density generators parameterized by a finite-dimensional nuisance parameter, and (iii) completely unknown density generators within a semiparametric setting.
In the latter case, we extend the semiparametric FIM to the CG subclass of ES distributions.
 Importantly, it also provides explicit analytical expressions for all associated correction terms, thereby enabling efficient computation and rigorous performance analysis in practical signal processing applications.

The remainder of this paper is organized as follows.
Section \ref{sec:Background statistical} introduces the fundamental concepts of multidimensional stationary CGPs and reviews the Slepian-Bangs formula for multivariate ES random variables.  
Section \ref{sec:Semiparametric FIM for CG distribution under constraint} extends the semiparametric FIM developed so far for constrained ES distributions  to constrained CG distributions.
Section \ref{sec:Properties of the associated coefficients of the Slepian-Bangs formula} establishes new analytical properties of the coefficients appearing in the Slepian-Bangs formula when specialized to CG-distributed random variables.
Section \ref{sec:FIM rate limit for stationary SIRP} presents the proof of a generalized  Whittle's formula for stationary CGPs with arbitrary  (non-zero) mean.
Section \ref{sec:Numerical illustrations}, provides numerical illustrations,  focusing on the Student's $t$ distributed first-order autoregressive (AR(1)) process. Finally, Section \ref{sec:Conclusion} concludes the paper.

The following notations are used throughout this paper.
Matrices and vectors are represented by bold upper case and bold lower case
characters, respectively. Vectors are by default in column
orientation, while the superscripts $T$, $H$ and $*$ stand for transpose,
conjugate transpose and conjugate, respectively.
The operators $\e(.)$, $|.|$, and $\tra(.)$ are the expectation, determinant, and trace operators, respectively. 
${\bf I}$, $\mathds{1}$ and ${\bf 1}$  denote, respectively, the identity matrix, the all-ones vector, and the indicator function,
 with the appropriate dimension. The symbols $\otimes$ is the Kronecker product of two matrices. ${\rm diag}(b_1,..,b_n)$ [resp., ${\rm Blkdiag}({\bf B}_1,...,{\bf B}_n)$] is the diagonal [resp., block diagonal] matrix with
$(b_1,..,b_n)$ [resp., ${\bf B}_1,...,{\bf B}_n$] along its diagonal. 
$[{\bf A}]_{k,\ell}$ and $[{\bf A}]_{[k,\ell]}$ denote the $(k,\ell)$-th entry and the $(k,\ell)$-th block of the block matrix ${\bf A}$, respectively.
The matrix norms $\|.\|_1$, $\|.\|_2$ and $\|.\|_{\infty}$ are the standard induced norms, while  $\|.\|_{\mathrm{Fro}}$ and $\|.\|_*$ denote the Frobenius and nuclear norm (sum of the singular values), respectively. We use $\lambda_{\min}({\bf S})$ to denote the smallest eigenvalue of ${\bf S}$.
%The Gamma function is defined as $\Gamma(u)\pardef \int_{0}^{\infty}t^{u-1}\exp(-t)dt$.
%The notation $x=_d y$  indicates equality in distribution between the random variables $x$ and $y$.
The acronyms r.v., p.d.f., c.d.f. and i.f.f. are used to represent random variable, probability density function, cumulative distribution function, and ''if and only if'', respectively.
%
%%%%%%%%%%%%%%%%%%%%%%%%%%%%%%%%%%%%%%%%%%%%%%%%%%%%%%%
%%%%%%%%%%%%%%%%%%%%%%%%%%%%%%%%%%%%%%%%%%%%%%%%%%%%%%%%%%%%%
\section{Statistical Background}
\label{sec:Background statistical}
%%%%%%%%%%%%%%%%%%%%%
%
This preliminary section introduces the concepts of real-valued multidimensional stationary
SIRP and recalls  essential theoretical results that underpin the development of  our main  contributions, which will be further explored in subsequent sections. 
%%%%%%%%%%%%%%%%%%%%%%%%%%%%%%%%%%%%%%%%
\subsection{Preliminaries on multidimensional SIRP}
\label{sec:Preliminaries on multidimensional SIRP}
%%%%%%%%%%%%%%%%%%%%%%%%%%%%%%%%%%%%%%%%%%%%%%%%%%%%%%%%
%
Generally, multidimensional discrete-time random processes $({\bf x}_k)_{k \in  \mathbb{Z}}$ with ${\bf x}_k \in \mathbb{R}^m$ can be defined by explicitly specifying the $n$th order p.d.f. of the r.v.
${\bf x}=({\bf x}_{k_1}^T,..,{\bf x}_{k_n}^T)^T$  for any $k_1<...<k_n$.  Well-known examples include processes with independent variables ${\bf x}_k$, 
 processes  with independent increments (i.e., ${\bf x}_{k_j}\!-\!{\bf x}_{k_i}$), Markov processes, and Gaussian processes (GPs). By analogy with GPs, it is natural to define elliptically symmetric (ES) processes from the joint distribution of ES r.v. ${\bf x}\in \mathbb{R}^{mn}$, whose p.d.f. is given by
\begin{equation}
\label{eq:pdf RES}
p({\bf x})
=
|\boldsymbol{\Sigma}|^{-1/2}g_{mn}[({\bf x}-\boldsymbol{\mu})^T \boldsymbol{\Sigma}^{-1}({\bf x}-\boldsymbol{\mu})],
\end{equation}
where $\e({\bf x})=\boldsymbol{\mu}$, $\boldsymbol{\Sigma}$  represents a symmetric positive definite matrix  and  $g_{mn}(.)$ (called the density generator) is an arbitrary function:
$\mathbb{R}^+ \mapsto \mathbb{R}^+$ such that $\int_0^{\infty}t^{mn/2-1}g_{mn}(t)dt=\delta_{mn}\pardef\pi^{-mn/2}\Gamma(mn/2)$ ensuring that $p({\bf x})$ integrates to one.
The density generator  $g_{mn}(.)$ may be either  (i) fully specified, (ii) parameterized by a finite-dimensional nuisance parameter, or (iii) left completely unspecified.
Furthermore, to avoid the scale ambiguity $(g_{mn}(.),\boldsymbol{\Sigma})$ in \eqref{eq:pdf RES}, it is assumed that the function $g_{mn}(.)$ is constrained to ensure that $\cov({\bf x})=\boldsymbol{\Sigma}$.

The  p.d.f. in \eqref{eq:pdf RES} can be derived as the p.d.f. of the following r.v. \cite{Wise1978}:
\begin{equation}
\label{eq:Stochastic representation ES}
{\bf x}
={\boldsymbol \mu}+ \sqrt{\mathcal{Q}_{mn}}{\boldsymbol{\Sigma}}^{1/2} {\bf u},
\end{equation}
where $\mathcal{Q}_{mn}$ is a positive-valued r.v. whose p.d.f. depends on $mn$ and is given by
\begin{equation}
\label{eq:Q pdf}
p_Q(q) = \delta_{mn}^{-1} q^{mn/2-1} g_{mn}(q),
\end{equation}
and ${\bf u}$ is uniformly distributed on the unit sphere in $\mathbb{R}^{mn}$ and is independent of $\mathcal{Q}_{mn}$. Consequently,
\begin{equation}
\label{eq:Q}
\mathcal{Q}_{mn}
=({\bf x}-{\boldsymbol \mu})^T{\boldsymbol{\Sigma}}^{-1}({\bf x}-{\boldsymbol \mu}).
\end{equation}
However, among the ES distributions defined by \eqref{eq:pdf RES}, only the subclass of compound Gaussian (CG) distributions satisfies the Kolmogorov consistency condition, according to which  the ES distribution is closed under marginalization, a necessary condition for defining consistently stochastic processes.
This allows us to define CGP from CG distributed r.v..
The CG distributions of ${\bf x}$ are characterized \cite[Th. 2.2]{Yao1973}, by the existence of a positive r.v $\tau$ (called texture of ${\bf x}$) of c.d.f. $F(\tau)$ such that $g_{mn}(t)$ defined by \eqref{eq:pdf RES} must to be written in the form:
\begin{equation}
\label{eq:density generator CG}
g_{mn}(t)
\!=\!
(2\pi)^{-mn/2}
\int_0^{\infty}\!
\tau^{-mn/2}\exp(-t/2\tau)
dF(\tau).
\end{equation}
The associated p.d.f. can be also derived from the p.d.f. of the following r.v. \cite{Andrews1974}: 
\begin{equation}
\label{eq:Stochastic representation CG}
{\bf x}
={\boldsymbol \mu}+ \sqrt{\tau}{\boldsymbol{\Sigma}}^{1/2}{\bf n}_0,
\end{equation}
where the texture r.v $\tau$ is independent of the r.v. ${\bf n}_0\sim{\cal N}_{mn}({\bf 0},{\bf I})$.
This representation is a particular case of \eqref{eq:Stochastic representation ES}, with
\begin{equation}
\label{eq:Stochastic representation CG b}
\mathcal{Q}_{mn}= \tau \chi^2_{mn}
\end{equation}
where $\tau$ and $\chi^2_{mn}$ are independent.

The representation \eqref{eq:Stochastic representation CG} emphasizes the hierarchical nature of the CG process: conditional on the texture r.v. $\tau$, the r.v. ${\bf x}$ is Gaussian with mean ${\boldsymbol \mu}$ and covariance ${\tau}{\boldsymbol{\Sigma}}$, whereas the randomness of $\tau$ gives rise to the heavy-tailed distribution of ${\bf x}$.
Under the normalization constraint $\cov({\bf x})=\boldsymbol{\Sigma}$, this further implies $\e(\mathcal{Q}_{mn})=mn$ and $\e(\tau)=1$.
The definition in \eqref{eq:Stochastic representation CG} precisely characterizes the subclass of CG distributions, whose
 associated discrete-time random processes are commonly
referred to as SIRP in the engineering literature. 
%
%%%%%%%%%%%%%%%%%%%%%%%%%%%%%%%%%%%%%%%%
\subsection{Preliminaries on multidimensional stationary SIRP}
\label{sec:Preliminaries on multidimensional stationary SIRP}
%%%%%%%%%%%%%%%%%%%%%%%%%%%%%%%%%%%%%%%%%%%%%%%%%%%%%%%%
%
According to the stochastic representation \eqref{eq:Stochastic representation CG} of  a CG r.v.,
the distribution of a non-zero mean CGP is characterized by the mean ${\boldsymbol \mu}$, the distribution
of the scalar r.v $\tau$ and by the covariance matrices $\boldsymbol{\Sigma}$ of the r.v.
${\bf x}=({\bf x}_{k_1}^T,..,{\bf x}_{k_n}^T)^T$  for any $ k_1 <...< k_n$, where $n\in\mathbb{N}^*$ and $k_i \in\mathbb{Z}$.

A CGP is by definition, stationary if the distribution of the r.v. $({\bf x}_{1+k},...,{\bf x}_{n+k})$ 
 is the same as the distribution of $({\bf x}_{1},...,{\bf x}_{n})$ for all $n\in \mathbb{N}^*$, $k \in \mathbb{Z}$.
Consequently, and analogous to Gaussian processes, a CGP is stationary  i.f.f. it is wide-sense stationary.
Precisely, a CGP is stationary i.f.f.
$\e({\bf x}_{\ell})$ and
$\e({\bf x}_{\ell}{\bf x}_{\ell+k}^T)$
are independent of $\ell\in \mathbb{Z},\ \forall k \in \mathbb{Z}$.
It follows that for a non-zero mean stationary CGP, the r.v. formed by stacking $n$ consecutive observations of ${\bf x}_k$ given by
\begin{equation}
\label{eq:yn}
{\bf y}_n\pardef({\bf x}_{1}^T,...,{\bf x}_{n}^T)^T \in \mathbb{R}^{mn}\ \
\forall  n \in \mathbb{N}^*,
\end{equation}
has mean
${\boldsymbol \mu}_{{\bf y}_n} \pardef\e({\bf y}_{n})
=(\mathds{1}_{n}\otimes {\bf I}_{m}){\boldsymbol \mu}_{x}$ with 
${\boldsymbol \mu}_{x}\pardef \e({\bf x}_{\ell})$ and covariance matrix $\boldsymbol{\Sigma}_{{\bf y}_n} \pardef \cov({\bf y}_n)$
which is symmetric positive definite and has a block Toeplitz structure with block size $m$.
Its $(i,j)$-th block is given by ${\bf R}_{x}(j-i)$ with 
${\bf R}_{x}(k) \pardef \e[({\bf x}_{\ell}-{\boldsymbol \mu}_x)({\bf x}_{\ell+k}-{\boldsymbol \mu}_x)^T] \in \mathbb{R}^{m \times m}$.

We assume the following assumptions, which are essential for the subsequent asymptotic analysis.
\begin{itemize}
\item
The sequence ${\bf R}_{x}(k)$ is absolutely summable, i.e., $\sum_{k=-\infty}^{\infty}\|{\bf R}_{x}(k)\|_2<\infty$. This ensures the existence of a constant $M>0$ such that $\|{\bf R}_{x}(k)\|_2<M$ for all $k$, and guarantees that the spectra  ${\bf S}_{x}(f) \pardef \sum_k {\bf R}_{x}(k)e^{-i2\pi kf}$ is well-defined and continuous on $[0,1]$.
\item
${\bf S}_{x}(f)$ is pointwise definite positive on $[0,1]$.
\end{itemize}
These assumptions imply two key properties:
\begin{itemize}
\item
The eigenvalues of ${\bf S}_{x}(f)$ are continuous on $[0,1]$; hence, there exists $c_1>0$ such that
$\lambda_{\min}({\bf S}_{x}(f))\ge c_1$ for all $f \in [0,1]$.
\item
$\boldsymbol{\Sigma}_{{\bf y}_n}^{-1} $ is uniformly bounded in spectral norm, i.e., there exists $c_r\!>\!0$ such that 
$\|\boldsymbol{\Sigma}_{{\bf y}_n}^{-1}\|_2 =\frac{1}{\lambda_{\min}(\boldsymbol{\Sigma}_{{\bf y}_n})}< c_r$ for all $n$.
\end{itemize}
We now introduce additional regularity conditions on the parameterization. Suppose that the mean ${\boldsymbol \mu}_x$ and the sequence ${\bf R}_{x}(k)$ 
depend on a parameter $\boldsymbol{\theta}=(\theta_1,..,\theta_r)^T \in \boldsymbol{\Theta}  \subset \mathbb{R}^r$, and consequently, so does the spectral density
${\bf S}_{x}(f)$.  For notational simplicity, this dependence on $\boldsymbol{\theta}$ is omitted whenever no ambiguity arises.
We assume the mapping
$\boldsymbol{\theta} \mapsto ({\boldsymbol \mu}_x,[{\bf R}_{x}(k)]_{k\in \mathbb{Z}},{\bf S}_{x}(f))$
is differentiable and locally one-to-one in a neighborhood of 
the true value of $\boldsymbol{\theta}$. 
Moreover, for each $\ell=1,...,r$, the derivative sequence
${\bf R}_{x,\ell}^{'}(k)\pardef \frac{\partial {\bf R}_{x}(k)}{\partial \theta_{\ell}}$ is absolutely summable at the true parameter value of $\boldsymbol{\theta}$. 
Consequently, there exist constants $M'_\ell$ and $c'_{\ell}$ such that $\|{\bf R}_{x,\ell}^{'}(k)\|_2 <M'_\ell$ for all $k$ and $\|{\bf S}_{x,\ell}^{'}(f)\|_2\le c'_{\ell}$ for all $f \in [0,1]$, where  ${\bf S}_{x,\ell}^{'}(f)\pardef \frac{\partial {\bf S}_{x}(f)}{\partial \theta_{\ell}}$.
%
%%%%%%%%%%%%%%%%%%%%%%%%%%%%%
\subsection{Reminder on Slepian-Bangs's formula for ES distributions}
\label{sec:Reminder on Slepian Bangs's formula for ES distributions}
%%%%%%%%%%%%%%%%%%%%%%%%%%%%%%%%%%%%%%%%%
%
We recall the FIM (also known as the Slepian-Bangs's formula) for a non-zero mean arbitrary ES distributed r.v.
${\bf y}_n \in \mathbb{R}^{mn}$ with p.d.f. given by \eqref{eq:pdf RES},
where the density generator $g_{mn}(.)$ is either fully known  \cite{Besson2013,Greco2013}, parameterized by a finite-dimensional nuisance parameter \cite{Abeida2023}, or completely unspecified \cite{Fortunati2026}.
The standard Fisher regularity conditions required to derive the FIM in these settings reduce, in the ES framework, to the single integrability condition
$\int_{0}^{\infty}\!\frac{1}{{g}_{mn}(t)}\left[\frac{dg_{mn}(t)}{dt}\right]^2t^{mn/2+1}dt\!<\!\infty$, which assumes that 
${g}_{mn}(t)$ is  differentiable and its derivative is continuous on $(0,\, \infty)$.
Dependent on on the level of knowledge about the density generator $g_{mn}(.)$, we consider the following three FIM:
\begin{itemize}
\item
${\bf I}_{{\bf y}_n}(\boldsymbol{\theta})$: FIM for known $g_{mn}(.)$  \cite{Besson2013,Greco2013}.
\item
$\bar{\bf I}_{{\bf y}_n}(\boldsymbol{\theta}|\boldsymbol{\alpha})
={\bf I}_{{\bf y}_n}(\boldsymbol{\theta})
-{\bf I}_{{\bf y}_n}(\boldsymbol{\theta},\boldsymbol{\alpha}){\bf I}_{{\bf y}_n}(\boldsymbol{\alpha})^{-1}
{\bf I}_{{\bf y}_n}^T(\boldsymbol{\theta},\boldsymbol{\alpha})$
%{\bf I}_{{\bf y}_n}^T(\boldsymbol{\theta,\boldsymbol{\alpha}})$: 
FIM for $g_{mn}^{\boldsymbol{\alpha}}(.)$ known up to the parameter $\boldsymbol{\alpha}$. The matrix
$\small{\left(\begin{array}{cc}
{\bf I}_{{\bf y}_n}(\boldsymbol{\theta})& {\bf I}_{{\bf y}_n}(\boldsymbol{\theta},\boldsymbol{\alpha})\\
{\bf I}_{{\bf y}_n}^T(\boldsymbol{\theta},\boldsymbol{\alpha})& {\bf I}_{{\bf y}_n}(\boldsymbol{\alpha})\\
\end{array}
\right)}$ denotes the joint FIM of the full parameter vector $(\boldsymbol{\theta}^T, \boldsymbol{\alpha}^T)^T$ for the parametrized p.d.f. $p({\bf y}_n$)
 \cite{Abeida2023}. This case requires the additional regularity conditions
$\int_{0}^{\infty}\frac{1}{{g}_{mn}^{\boldsymbol{\alpha}}(t)}\left[ \frac{dg_{mn}^{\boldsymbol{\alpha}}(t)}{dt}\right]
\left[\nabla_{\boldsymbol{\alpha}}g_{mn}^{\boldsymbol{\alpha}}(t)\right]t^{mn/2}dt<\infty$
and
$\int_{0}^{\infty}\left[\nabla_{\boldsymbol{\alpha}}g_{mn}^{\boldsymbol{\alpha}}(t)\right]
\left[\nabla_{\boldsymbol{\alpha}}g_{mn}^{\boldsymbol{\alpha}}(t)\right]^Tt^{mn/2}dt<\infty$.
\item
$\bar{\bf I}_{{\bf y}_n}(\boldsymbol{\theta}|\bar{g}_{mn})$: semiparametric FIM for ES distributions, treating $g_{mn}(.)$  as an unknown infinite-dimensional nuisance parameter, under the constraint $\e(\mathcal{Q}_{mn})=mn$ with $\e(\mathcal{Q}_{mn}^2)<\infty$ \cite{Fortunati2026}.
\end{itemize}
Interestingly, all three forms share the same structural representation as that of  
the corresponding elementwise FIM expression:
\begin{eqnarray}
\nonumber
\left[{\rm FIM}_{{\bf y}_n}(\boldsymbol{\theta})\right]_{k,\ell}
\!\!\!\!&=&\!\!\!\!
a_{0,mn}{\boldsymbol \mu'}^T_{{\bf y}_n,k}\ \boldsymbol{\Sigma}_{{\bf y}_n}^{-1}{\boldsymbol \mu}'_{{{\bf y}_n},\ell}
\\
\nonumber
\!\!\!\!&+&\!\!\!\!
a_{1,mn}
 \tra(\boldsymbol{\Sigma}_{{\bf y}_n}^{-1}\boldsymbol{\Sigma}_{{\bf y}_n,k}^{'}
 \boldsymbol{\Sigma}_{{\bf y}_n}^{-1}\boldsymbol{\Sigma}_{{\bf y}_n,\ell}^{'})
\\
\label{eq:Slepian-Bangs's formula}
\!\!\!\!&+&\!\!\!\!
a_{2,mn}\tra(\boldsymbol{\Sigma}_{{\bf y}_n}^{-1}\boldsymbol{\Sigma}_{{\bf y}_n,k}^{'})\tra(\boldsymbol{\Sigma}_{{\bf y}_n}^{-1}\boldsymbol{\Sigma}_{{\bf y}_n,\ell}^{'}),
\end{eqnarray}
where ${\boldsymbol \mu}'_{{\bf y}_n,k}\pardef \frac{\partial {\boldsymbol \mu}_{{\bf y}_n}}{\partial \theta_k}$
and  
$\boldsymbol{\Sigma}_{{\bf y}_n,k}^{'}\pardef \frac{\partial \boldsymbol{\Sigma}_{{\bf y}_n}}{\partial \theta_k}$
and 
with identical coefficients $a_{0,mn}$ and  $a_{1,mn}$, while differing only in the value of $a_{2,mn}$, which captures the uncertainty associated with the density generator.
The coefficients $a_{0,mn}$ and $a_{1,mn}$ are given by
\begin{equation}
\label{eq:a0a1}
a_{0,mn}=\xi_{1,mn}
\ \mbox{and}\
a_{1,mn}=\frac{1}{2}\xi_{2,mn}
\end{equation}
with\footnote{We note that 
$\xi_{2,mn}$ is free of scale ambiguity in contrast to $\xi_{1,mn}$.}
\begin{eqnarray}
\label{eq:def xi1r}
\xi_{1,mn}
&\pardef&
\frac{\e[\mathcal{Q}_{mn}\varphi^2(\mathcal{Q}_{mn})]}{mn}
\\
\label{eq:def x2r}
\xi_{2,mn}
&\pardef&
\frac{\e[\mathcal{Q}_{mn}^2\varphi^2(\mathcal{Q}_{mn})]}{mn(mn+2)},
\end{eqnarray}
where
\begin{equation}
\label{eq:dephi}
\varphi(t)
\pardef   
-\frac{2}{g_{mn}(t)}\frac{d g_{mn}(t)}{dt}.
\end{equation}
Regarding the coefficient $a_{2,mn}$, we have the following three cases:
\begin{eqnarray}
\label{eq:a2class}
a_{2,mn}^{\rm Clas}
&=&
\frac{1}{4}(\xi_{2,mn}\!-\!1) \ \mbox{for known}\ g_{mn},
\\
\nonumber
a_{2,mn}^{\rm Pa}
\!\!\!\!&=&\!\!\!\!
a_{2,mn}^{\rm Clas}- {\boldsymbol \xi}_{3,mn}^T \Xi_{4,mn}^{-1}{\boldsymbol \xi}_{3,mn}\ 
\\
\label{eq:a2clpa}
&&\hspace{1.8cm}
\mbox{for}\ g_{mn}\ \mbox{known up to}\ {\boldsymbol \alpha},   
\\
\nonumber
a_{2,mn}^{\rm SePa}
\!\!\!\!&=&\!\!\!\!
\frac{1}{\sigma^2_{\mathcal{Q}_{mn}}}\!-\! \frac{1}{2mn}\xi_{2,mn}\ \mbox{for unknown}\ g_{mn}
\\
\label{eq:a2clsepa}
&&
\mbox{under the constraint}\ \e(\mathcal{Q}_{mn})=mn,
\end{eqnarray}
where ${\boldsymbol \xi}_{3,mn}\pardef \frac{\e[\mathcal{Q}_{mn}\varphi(Q_{mn}){\boldsymbol \varphi}^{\boldsymbol \alpha}(\mathcal{Q}_{mn})]}{mn},
\Xi_{4,mn}\pardef \e[{\boldsymbol \varphi}^{\alpha}(\mathcal{Q}_{mn}){{\boldsymbol \varphi}^{\boldsymbol \alpha}}^T(\mathcal{Q}_{mn})]$
with
${\boldsymbol \varphi}^{\boldsymbol \alpha}(t)\pardef - \frac{1}{g_{mn}^{\boldsymbol \alpha}(t)}\nabla_{\boldsymbol{\alpha}}g_{mn}^{\boldsymbol{\alpha}}(t)$, and $\sigma^2_{\mathcal{Q}_{mn}}\pardef \var(\mathcal{Q}_{mn})$.
Finally, for the Gaussian distribution, which corresponds to $\varphi(t) =1$, we recover the classical values  $(a_{0,mn},a_{1,mn},a_{2,mn}^{\rm Clas})=(1,1/2,0)$. 
%
%%%%%%%%%%%%%%%%%%%%%%%%%%%%%%%%%%%%%%%%%%%%%%%%%%%%%%%%%%%%%%%%%%%%%%%%%%%%%%%
\section{Semiparametric FIM for CG Distributions Under the Constraint $\e(\tau)=1$}
\label{sec:Semiparametric FIM for CG distribution under constraint}
%%%%%%%%%%%%%%%%%%%%%%%%%%%%%%%%%%%%%%%%%%%%%%
%
The broader context within which CG distributions can be placed is that of
semiparametric models. For an in-depth discussion of semiparametric theory,
we refer the reader to the seminal book \cite{Bickel1993}. Here, without any
claim to completeness, we confine ourselves to those points necessary for the
continuation of our study. Specifically, the purpose of this section is to
extend the semiparametric FIM derived in \cite{Fortunati2026} for ES
distributions under the constraint $\e(\mathcal{Q}_{mn})=mn$ to CG
distributions under the constraint $\e(\tau)=1$. Although the two constraints
are equivalent within the CG class, this extension is not a mere
particularization, since the nuisance parameterizations of the two models
differ. Moreover, a priori, the additional structural information that the
data belong to the CG subclass should \emph{increase} the semiparametric FIM
relative to that provided in \cite{Fortunati2026}; that this does not occur,
as shown by Result~\ref{Semiparametric FIM CG}, is therefore a nontrivial
fact.

Let $\mathcal{F}$ denote the set of all c.d.f.s of the (continuous, discrete,
or mixed) texture r.v.\ $\tau$. In our study, we consider the subset
\begin{equation}
\label{set_F_c}
\overline{\mathcal{F}} 
=\left\{\bar{F} \in \mathcal{F}  \left|  \int_{0}^{\infty}t\ d\bar{F}(t)  =\e(\tau) = 1\right. \right\}.
\end{equation}
Let us now introduce the quadratic form parameterized by
$\boldsymbol{\theta}\in\boldsymbol{\Theta}$:
\begin{equation}
\label{eq:Q_alpha}
\mathcal{Q}_{\boldsymbol{\theta}}({\bf y}_n)
=({\bf y}_n-{\boldsymbol \mu}_{{\bf y}_n}(\boldsymbol{\theta}))^T
{\boldsymbol{\Sigma}_{{\bf y}_n}^{-1}(\boldsymbol{\theta})}
({\bf y}_n-{\boldsymbol \mu}_{{\bf y}_n}(\boldsymbol{\theta})),
\end{equation}
for all $\boldsymbol{\theta}\in\boldsymbol{\Theta}$. The \textit{constrained}
semiparametric CG model is then defined as
\begin{equation}
\label{eq:c_CG_sempar_model}
{\mathcal{P}}^c_{\boldsymbol{\theta},\bar{F}}
\!=\!\left\{ p({\bf y}_n|\boldsymbol{\theta},\bar{F})
\!=\!\int_{0}^{\infty}\!\! p({\bf y}_n|\boldsymbol{\theta},t)\,d\bar{F}(t):
\boldsymbol{\theta}\in \boldsymbol{\Theta},\ \bar{F}\in \overline{\mathcal{F}}
\right\},
\end{equation}
where, from \eqref{eq:pdf RES} and \eqref{eq:density generator CG},
\begin{equation}
\label{eq:def_C}
\begin{split}
p({\bf y}_n|\boldsymbol{\theta},t)
&= (2\pi t)^{-mn/2} |\boldsymbol{\Sigma}_{{\bf y}_n}(\boldsymbol{\theta})|^{-1/2}
\exp\!\left(-\mathcal{Q}_{\boldsymbol{\theta}}({\bf y}_n)/2t\right)\\
&= \mathcal{N}({\boldsymbol \mu}_{{\bf y}_n}(\boldsymbol{\theta}),\,
t\,\boldsymbol{\Sigma}_{{\bf y}_n}(\boldsymbol{\theta})).
\end{split}
\end{equation}

In the semiparametric framework, only the finite-dimensional parameter
$\boldsymbol{\theta}\in\boldsymbol{\Theta}$ is of interest, while the
functional (infinite-dimensional) parameter
$\bar{F}\in\overline{\mathcal{F}}$ is treated as a nuisance. It is therefore
of interest to compute the so-called semiparametric efficient FIM
$\bar{\bf I}(\boldsymbol{\theta}|\bar{F})$, which, roughly speaking, is the
FIM for $\boldsymbol{\theta}$ in the presence of the unknown nuisance c.d.f.\
$\bar{F}$. As discussed in \cite[Sect.~3.3]{Bickel1993}, evaluating
$\bar{\bf I}(\boldsymbol{\theta}|\bar{F})$ requires three ingredients:
\begin{itemize}
\item
the (infinite-dimensional) nuisance tangent space
$\mathcal{T}^c_{\bar{F}}$;
\item
the orthogonal projection operator onto $\mathcal{T}^c_{\bar{F}}$, denoted
$\Pi(\cdot|\mathcal{T}^c_{\bar{F}})$;
\item
the semiparametric efficient score vector
$\bar{\bf s}_{\boldsymbol{\theta}}\equiv\bar{\bf s}_{\boldsymbol{\theta}}({\bf y}_n)$,
defined as
\begin{equation}
\label{eq:efficient score}
\bar{\bf s}_{\boldsymbol{\theta}} \pardef {\bf s}_{\boldsymbol{\theta}} -
\Pi({\bf s}_{\boldsymbol{\theta}}|\mathcal{T}^c_{\bar{F}}),
\end{equation}
where ${\bf s}_{\boldsymbol{\theta}} \pardef
\nabla_{\boldsymbol{\theta}}\ln p({\bf y}_n|\boldsymbol{\theta},\bar{F})$ is
the standard score vector evaluated at the c.d.f.\ $\bar{F}$.
\end{itemize}
Consequently, the semiparametric efficient FIM is defined as
\cite[Sect.~3.3]{Bickel1993}
\begin{equation}
\bar{\bf I}(\boldsymbol{\theta}|\bar{F}) \pardef
\e\!\left(\bar{\bf s}_{\boldsymbol{\theta}}
\bar{\bf s}_{\boldsymbol{\theta}}^T\right).
\end{equation}
Readers interested in the mathematical derivation of the aforementioned
ingredients will find it in the supplemental material. In this section, we
confine ourselves to presenting the final result:
\begin{result}
\label{prop_c}
Let $\mathcal{P}^c_{\boldsymbol{\theta},\bar{F}}$ in
\eqref{eq:c_CG_sempar_model} be the constrained CG semiparametric model.
Then, the (infinite-dimensional) nuisance tangent space
$\mathcal{T}^c_{\bar{F}}$ at $\bar{F} \in \overline{\mathcal{F}}$ is given by
\begin{equation}
\label{eq:Tc_G0_1}
\begin{split}
\mathcal{T}^c_{\bar{F}}
&= \left\lbrace h : \mathbb{R}^{mn} \rightarrow \mathbb{R} |
h\ \text{is $\sigma(\mathcal{Q}_{mn})$-measurable},\right. \\
& \qquad \left. \e(h(\mathcal{Q}_{mn})) =0,\;
\e(\mathcal{Q}_{mn}h(\mathcal{Q}_{mn}) )=0
\right\rbrace,
\end{split}
\end{equation}
where $\sigma(\mathcal{Q}_{mn})\subset \mathfrak{B}(\mathcal{X})$ is the
sub-$\sigma$-algebra generated by the r.v.\
$\mathcal{Q}_{mn} =Q_{\boldsymbol{\theta}}({\bf y}_n)$ in \eqref{eq:Q_alpha},
with $\boldsymbol{\theta}$ here denoting the true parameter vector.
Moreover, the orthogonal projection of ${\bf s}_{\boldsymbol{\theta}}$ onto
$\mathcal{T}^c_{\bar{F}}$ is given by
\begin{equation}
\label{eq:proj_cg}
\Pi({\bf s}_{\boldsymbol{\theta}}|\mathcal{T}^c_{\bar{F}})
= \e({\bf s}_{\boldsymbol{\theta}}|\mathcal{Q}_{mn})
- \e(\mathcal{Q}_{mn}{\bf s}_{\boldsymbol{\theta}})\,
\sigma_{\mathcal{Q}_{mn}}^{-2}\,(\mathcal{Q}_{mn}-mn).
\end{equation}
\end{result}
Comparing \eqref{eq:proj_cg} with the orthogonal projection
\cite[eq.~(81)]{Fortunati2026}, we immediately obtain the following:
\begin{result}
\label{Semiparametric FIM CG}
The semiparametric efficient FIM
$\bar{\bf I}_{{\bf y}_n}(\boldsymbol{\theta}|\bar{F})$ for CG distributions
under the constraint $\e(\tau)=1$ is identical to the semiparametric
efficient FIM $\bar{\bf I}_{{\bf y}_n}(\boldsymbol{\theta}|\bar{g}_{mn})$
for ES distributions under the constraint $\e(\mathcal{Q}_{mn})=mn$, given
in \cite[eq.~(116)]{Fortunati2026} and, here, by
\eqref{eq:Slepian-Bangs's formula} with \eqref{eq:a0a1} and
\eqref{eq:a2clsepa}.
\end{result}
In other words, although restricting the ES model to the CG subclass reduces the infinite-dimensional nuisance space, this additional structural knowledge provides no further information about the finite dimensional parameter $\boldsymbol{\theta}$. Hence, the semiparametric CRB derived in \cite{Fortunati2026} remains the relevant performance bound for CG data.
%
%%%%%%%%%%%%%%%%%%%%%%%%%%%%%%%%%%%%%%%%%%%%%%%%%%%%%%%%%%%%%%%%%%%%%%%%%%%%%%%%%%%%%%%%%%%%%%%%%%%%%%%%%%%%%%%%%%%%%%%%%%%%%%%%%%%%%
\section{Properties of the associated coefficients of the Slepian-Bangs formula}
\label{sec:Properties of the associated coefficients of the Slepian-Bangs formula}
%%%%%%%%%%%%%%%%%%%%%%%%%%%%%%%%%%%%%%%%%
%%%%%%%%%%%%%%%%%%%%%%%%%%%%%%%%%%%%%%%%%%%%%%%%%%%%%%%%%%%%
%
For arbitrary ES distribution, there are no general properties for the coefficients $\xi_{1,mn}$ and $\xi_{2,mn}$ except the inequalities:
\begin{equation}
\label{eq:inequalities a0a2}
\xi_{1,mn}\ge 1
\ \ \mbox{and}\ \
\xi_{2,mn} \ge \frac{mn}{mn+2}
\end{equation}
derived from the respective Cauchy-Schwarz inequalities:
$(\e[\mathcal{Q}_{mn}^{1/2}.\mathcal{Q}_{mn}^{1/2}\varphi(\mathcal{Q}_{mn}])^2\le \e[\mathcal{Q}_{mn}] \e[\mathcal{Q}_{mn}\varphi^2(\mathcal{Q}_{mn})]$ and
$(\e[1.\mathcal{Q}_{mn}\varphi(\mathcal{Q}_{mn}])^2\le \e[1^2] \e[\mathcal{Q}_{mn}^2\varphi^2(\mathcal{Q}_{mn})]$
(using $\e[\mathcal{Q}_{mn}\varphi(\mathcal{Q}_{mn})]=mn$ proved in \cite{Abeida2019}).

We now present  additional novel analytical results concerning the coefficients  $a_{i,mn}$, $i=0,1,2$ in the Slepian-Bangs's formula \eqref{eq:Slepian-Bangs's formula} specialized to CG distributions. These results will play a key role in the developments of
\ref{sec:Asymptotic properties of the sequence of coefficients}.
In particular, by exploiting the representation in \eqref{eq:Stochastic representation CG b}, the following lemma is established in the 
 supplemental material:
\begin{lemma}
\label{property phi}
The r.v. $\varphi(\mathcal{Q}_{mn})$, which appears in the definition \eqref{eq:def xi1r} and \eqref{eq:def x2r}
of $\xi_{1,mn}$ and $\xi_{2,mn}$, respectively, can be expressed through the following conditional expectation:
\begin{equation}
\label{eq:phi q}
\varphi(\mathcal{Q}_{mn})
=
\e[\tau^{-1}|\mathcal{Q}_{mn}].
\end{equation}
\end{lemma}
This lemma \ref{property phi} enables us to prove the following lemma\footnote{$\var(y| x)$ denotes the conditional variance of  the r.v. $y$ given $x$, i.e., 
$\var(y|x) \pardef \e\{[y-\e(y|x)]^2|x\}$.} in the  supplemental material:
\begin{lemma}
\label{property xi 2}
\begin{equation}
\label{eq:xi 2}
\xi_{2,mn}
=
1- \frac{\e[\var(\chi_{mn}^2|\mathcal{Q}_{mn})]}{mn(mn+2)},
\end{equation}
\end{lemma}
\noindent
which in turn allows us to refine the second inequality in \eqref{eq:inequalities a0a2} as:
\begin{equation}
\label{eq:inequalities xi2}
 \frac{mn}{mn+2} \le \xi_{2,mn} \le 1.
\end{equation}
Finally, one may expect that more knowledge about the density generator $g_{mn}(.)$ yields a larger (in the sense of positive definite matrices ordering) FIM for 
$\boldsymbol{\theta}$, i.e.,
\begin{equation}
\label{eq:inequalities FIM}
\bar{\bf I}_{{\bf y}_n}(\boldsymbol{\theta}|\bar{F})
\le
\bar{\bf I}_{{\bf y}_n}(\boldsymbol{\theta}|\boldsymbol{\alpha})
\le
{\bf I}_{{\bf y}_n}(\boldsymbol{\theta}).
\end{equation}
Because the FIM in \eqref{eq:Slepian-Bangs's formula} is expressed as a linear combination of three positive definite matrices 
${\bf M}_{0,n}=\frac{\partial {\boldsymbol \mu}_{{\bf y}_n}^T}{\partial {\boldsymbol \theta}^T} \boldsymbol{\Sigma}_{{\bf y}_n}^{-1} \frac{\partial {\boldsymbol \mu}_{{\bf y}_n}}{\partial {\boldsymbol \theta}^T}$, 
${\bf M}_{1,n}=\frac{\partial \ve^T( \boldsymbol{\Sigma}_{{\bf y}_n})}{\partial {\boldsymbol \theta}^T}\boldsymbol{\Sigma}_{{\bf y}_n}^{-T}\otimes \boldsymbol{\Sigma}_{{\bf y}_n}^{-1}\frac{\partial \ve( \boldsymbol{\Sigma}_{{\bf y}_n})}{\partial {\boldsymbol \theta}^T}$ and 
${\bf M}_{2,n}=\frac{\partial \ve^T( \boldsymbol{\Sigma}_{{\bf y}_n})}{\partial {\boldsymbol \theta}^T}\ve(\boldsymbol{\Sigma}_{{\bf y}_n}^{-1})\ve^T(\boldsymbol{\Sigma}_{{\bf y}_n}^{-1})\frac{\partial \ve( \boldsymbol{\Sigma}_{{\bf y}_n})}{\partial {\boldsymbol \theta}^T}$,  inequality \eqref{eq:inequalities FIM} is equivalent to 
\begin{equation}
\label{eq:inequalities a2}
a_{2,mn}^{\rm SePa}
\le
a_{2,mn}^{\rm Pa}
\le
a_{2,mn}^{\rm Clas}.
\end{equation}
The inequality $a_{2,mn}^{\rm Pa}
\le
a_{2,mn}^{\rm Clas}$ follows directly from \eqref{eq:a2clpa}. In contrast, establishing the inequality $a_{2,mn}^{\rm SePa}
\le
a_{2,mn}^{\rm Pa}$, appears significantly more involved. Nevertheless, the following result is established in the Appendix
\ref{Proof of Result property FIM inequality}:
\begin{result}
\label{property FIM inequality}
For arbitrary CG distribution, 
\begin{equation}
\label{eq:inequality a2 clas sepa}
a_{2,mn}^{\rm SePa}
\le
a_{2,mn}^{\rm Clas},
\end{equation}
with equality i.f.f. the distribution is Gaussian.
\end{result}
%
%%%%%%%%%%%%%%%%%%%%%%
%%%%%%%
\section{FIM rate limit for stationary SIRP}
\label{sec:FIM rate limit for stationary SIRP}
%%%%%%%%%%%%%%%%%%%%%%%%%%%%%%%%%%%%%%%%%
%
This section is devoted to analyzing the asymptotic behavior of the FIM rate $\frac{1}{n}{\rm FIM}_{{\bf y}_n}(\boldsymbol{\theta})$ as 
$n\rightarrow \infty$, based on the closed-form expression of the FIM in
 \eqref{eq:Slepian-Bangs's formula} 
associated with $n$ consecutive observations of stationary SIRPs.
To this end, note that in \eqref{eq:Slepian-Bangs's formula},
the coefficients $a_{i,mn}, i=0,1,2$ depend solely on the distribution of the texture r.v. $\tau$, whereas the matrices ${\bf M}_{i,n}$, $i=0,1,2$ depend solely on the parameterized moments ${\boldsymbol \mu}_{{\bf y}_n}$ and ${\boldsymbol \Sigma}_{{\bf y}_n}$.
Moreover, $\boldsymbol{\Sigma}_{{\bf y}_n}$  exhibits a block-Toeplitz structure whose first block row is $({\bf R}_x(0),{\bf R}_x(1),..,{\bf R}_x(n\!-\!1))$.
Consequently, the asymptotic analysis of the FIM rate reduces to separately studying the asymptotic behavior of the coefficient sequences 
$a_{i,mn}$ and the matrix sequences
${\bf M}_{i,n}$, $i=0,1,2$.
%
%%%%%%%%%%%%%%%%%%%%%%%%%%%%%%%%%%%%%%%
\subsection{Asymptotic properties of the sequence of matrices ${\bf M}_{i,n}$}
\label{sec:General case for M}
%%%%%%%%%%%%%%%%%%%%%%%%%%%%%%%%%%%%%%
%
Let us first consider the study of the sequence $[{\bf M}_{0,n}]_{k,\ell}
={\boldsymbol \mu}_{x,k}^T(\mathds{1}_{n}^T\otimes {\bf I}_{m})
\boldsymbol{\Sigma}_{{\bf y}_n}^{-1}
(\mathds{1}_{n}\otimes {\bf I}_{m}){\boldsymbol \mu}_{x,\ell}$, for which, the following result is proven in the Appendix
\ref{Proof of Result Result matrice mean limit}:
\begin{result}
\label{Result matrice mean limit}
Under the assumptions given in Section \ref{sec:Preliminaries on multidimensional stationary SIRP}, we have the limit:
\begin{equation}
\label{eq: matrice mean limit}
\lim_{n \rightarrow \infty}\frac{1}{n}{\boldsymbol \mu'}^T_{{\bf y}_n,k}\ \boldsymbol{\Sigma}_{{\bf y}_n}^{-1}{\boldsymbol \mu}'_{{{\bf y}_n},\ell}
={\boldsymbol \mu'}^T_{{\bf x}_k}{\bf S}_{x}^{-1}(0){\boldsymbol \mu}'_{{\bf x}_\ell}.
\end{equation}
\end{result}
We now investigate the asymptotic behavior of the sequences $[{\bf M}_{1,n}]_{k,\ell}$ 
and $[{\bf M}_{2,n}]_{k,\ell}$. To this end, we use the notion of asymptotically equivalent matrix sequences, denoted by $\sim$, originally introduced by Gray~\cite{Gray2006} for scalar Toeplitz matrices and later extended to block Toeplitz matrices by Guti\'{e}rrez~\cite{Gutierrez2011}. This framework provides a convenient connection to the classical Szeg\"o's theory for multivariate time series \cite{Grenander1958}.
For completeness, we recall the definition:
\begin{definition}
For $d_n^1, d_n^2$ strictly increasing, $d_n^1 \times d_n^2$ matrices ${\bf A}_n$ and ${\bf B}_n$,
${\bf A}_n \sim {\bf B}_n$ means that:
$\|{\bf A}_n\|_2, \|{\bf B}_n\|_2 \le M <\infty$ and $\lim_{n\rightarrow \infty}\frac{\|{\bf A}_n-{\bf B}_n\|_{\rm Fro}}{\sqrt{n}}=0$ \cite[Definition 3.1]{Gutierrez2011}.
\end{definition}
Following Gray's notation, let
${\bf T}_n({\bf S}_{x}(f)) \pardef \boldsymbol{\Sigma}_{{\bf y}_n}$ denotes the $mn \times mn$ block Toeplitz matrix generated by the spectral density ${\bf S}_x(f)$.
We also define the associated block circulant matrix  with $m \times m$ blocks as
\begin{equation}
\label{eq:Def C}
{\bf C}_n({\bf S}_{x}(f))
\pardef
{\bf W}_{n,m}
{\rm Diag}({\bf S}_{x,n})
{\bf W}_{n,m}^H
\end{equation}
where
${\bf W}_{n,m}  \pardef {\bf W}_n \otimes {\bf I}_m$ and ${\bf W}_n \in \mathbb{C}^{n \times n}$ is the discrete Fourier transform unitary symmetric matrix  defined by
$[{\bf W}_n]_{k,\ell}\pardef \frac{1}{\sqrt{n}}e^{i2\pi (k-1)(\ell-1)/n},\ k,\ell = 1,\dots,n.$ 
The block-diagonal matrix $\mathrm{Diag}({\bf S}_{x,n})$ is given by
\begin{equation}
\label{eq:def Diag}
{\rm Diag}({\bf S}_{x,n})\!\!\pardef\!\!{\rm Blkdiag}({\bf S}_{x}(0),{\bf S}_{x}(\frac{1}{n}),..
.,{\bf S}_{x}(\frac{n\!-\!1}{n})).
\end{equation}
With these tools, the following result is proven in the Appendix
\ref{Proof of Result Result matrice variance limit}:
\begin{result}
\label{Result matrice variance limit}
Under the assumptions given in Section \ref{sec:Preliminaries on multidimensional stationary SIRP}, we have the limits:
\begin{equation}
\label{eq: matrice variance limit 1}
\lim_{n \rightarrow \infty}\frac{1}{n}
\tra(\boldsymbol{\Sigma}_{{\bf y}_n}^{-1}\boldsymbol{\Sigma}_{{\bf y}_n,k}^{'})
=
\int_{0}^{1}
\tra({\bf S}_{x}^{-1}(f) {\bf S}_{x,k}^{'}(f))df,
\end{equation}
\begin{eqnarray}
\nonumber
\lim_{n \rightarrow \infty}\frac{1}{n}
\tra(\boldsymbol{\Sigma}_{{\bf y}_n}^{-1}\boldsymbol{\Sigma}_{{\bf y}_n,k}^{'}\boldsymbol{\Sigma}_{{\bf y}_n}^{-1}
\boldsymbol{\Sigma}_{{\bf y}_n,\ell}^{'})
&=&
\label{eq: matrice variance limit 2}
\\
&&
\hspace{-4.5cm}
\int_{0}^{1}
\tra({\bf S}_{x}^{-1}(f) {\bf S}_{x,k}^{'}(f) {\bf S}_{x}^{-1}(f){\bf S}_{x,\ell}^{'}(f))df.
\end{eqnarray}
\end{result}
To quantify the rate of convergence of the limits in \eqref{eq: matrice mean limit}, \eqref{eq: matrice variance limit 1} and \eqref{eq: matrice variance limit 2}, we now consider the special case of  banded block-Toeplitz structured covariance
$\boldsymbol{\Sigma}_{\mathbf{y}_n}$ for which the following result is proven in the Appendix
\ref{Proof of Result Result banded Toeplitz}:
\begin{result}
\label{Result banded Toeplitz}
For  banded block-Toeplitz structured covariance $\boldsymbol{\Sigma}_{{\bf y}_n}$ satisfying ${\bf R}_x(k)={\bf 0}$ for $|k|>q$, 
the following upper bounds, which depend on ${\boldsymbol \theta}$, hold:
\begin{eqnarray}
\nonumber
|\frac{1}{n}{\boldsymbol \mu'}^T_{{\bf y}_n,k}\ \boldsymbol{\Sigma}_{{\bf y}_n}^{-1}{\boldsymbol \mu}'_{{{\bf y}_n},\ell}
-
{\boldsymbol \mu'}^T_{{\bf x}_k}{\bf S}_{x}^{-1}(0){\boldsymbol \mu}'_{{\bf x}_\ell}|
&&
\\
\label{eq:Result banded Toeplitz1}
&&
\hspace{-5cm}
\le\frac{h(q)c_rM}{\sqrt{n}}    \| {\bf S}_x^{-1}(0)\|_2\|{\boldsymbol \mu}'_{{\bf x}_k}\|_2\|{\boldsymbol \mu}'_{{\bf x}_\ell}\|_2,
\end{eqnarray}
where $h(q)\pardef \sqrt{q(q+1)(2q+1)/3}$.
\begin{eqnarray}
\nonumber
|\frac{1}{n}\tra(\boldsymbol{\Sigma}_{{\bf y}_n}^{-1}\boldsymbol{\Sigma}_{{\bf y}_n,k}^{'})
-
\int_{0}^{1}
\tra({\bf S}_{x}^{-1}(f) {\bf S}_{x,k}^{'}(f))df|
\\
\label{eq:Result banded Toeplitz2}
&&
\hspace{-5.5cm}
\le \frac{2q^2m c_r(M\frac{c'_k}{c_1}+M'_k)}{n}+\frac{\delta_{n,S,S_k}}{n}
\end{eqnarray}
where $\frac{\delta_{n,S,S_k}}{n}$ bounds the Riemann sum approximation error:
$|\frac{1}{n}
\tra\{{\rm Diag}
({\bf S}_{x,n}^{-1} {\bf S}_{x,n,k}^{'})\}
-\int_{0}^{1}
\tra({\bf S}_{x}^{-1}(f) {\bf S}_{x,k}^{'}(f))df|$.
\end{result}
Following the same derivation steps as above, the covariance-related term $\frac{1}{n} \mathrm{tr}(\boldsymbol{\Sigma}_{\mathbf{y}_n}^{-1} \boldsymbol{\Sigma}_{\mathbf{y}_n,k}' \boldsymbol{\Sigma}_{\mathbf{y}_n}^{-1} \boldsymbol{\Sigma}_{\mathbf{y}_n,\ell}')$, exhibits an
 analogous convergence rate $\mathcal{O}(n^{-1})$ to that of 
 $\frac{1}{n} \mathrm{tr}(\boldsymbol{\Sigma}_{\mathbf{y}_n}^{-1} \boldsymbol{\Sigma}_{\mathbf{y}_n,k}')$.

Result~\ref{Result banded Toeplitz} establishes distinct convergence rates for the two classes of matrix terms appearing in the Slepian-Bangs formula: 
the mean-related term converges at rate $\mathcal{O}(n^{-1/2})$, whereas the two covariance-related terms converge faster, at rate $\mathcal{O}(n^{-1})$. 
Moreover, the bounds in~\eqref{eq:Result banded Toeplitz1} and~\eqref{eq:Result banded Toeplitz2} reveal that the correlation length $q$ plays a critical role in the speed of convergence: the shorter the correlation length $q$, the faster the convergence of both sequences.
This quantitative result provides rigorous theoretical justification for the widely used approximate eigenvalue decomposition
\%
\begin{equation}
\label{eq:EVD}
\mathbf{W}_{n,m}^H \boldsymbol{\Sigma}_{\mathbf{y}_n} \mathbf{W}_{n,m} \approx \mathrm{Diag}(\mathbf{S}_{x,n}),
\end{equation}
which was previously employed heuristically in the literature (see, e.g.,~\cite[Eq.~(9)]{Zeira1990}, \cite[Eq.~(5)]{Porat1995}, and~\cite[p.~186]{Therrien1992}) without formal error bounds. Our analysis confirms that this approximation becomes increasingly accurate as $n \gg q$, with explicit error decay rates.
%
%%%%%%%%%%%%%%%%%%%%%%%%%%%%%%%%%%%%%%%
\subsection{Asymptotic properties of the sequence of coefficients $a_{i,mn}$}
\label{sec:Asymptotic properties of the sequence of coefficients}
%%%%%%%%%%%%%%%%%%%%%%%%%%%%%%%%%
%
From the previous subsection, the asymptotic behavior of ${\rm FIM}_{{\bf y}_n}(\boldsymbol{\theta})$
hinges on the limits of the three sequences $a_{0,mn}$, $a_{1,mn}$, and $n a_{2,mn}$, for which we prove the following result.
\begin{result}
\label{limit of the three coefficients}
For an arbitrary distribution of the texture r.v. $\tau$ whether continuous, discrete, or mixed, subject to the
%potentially containing both a discrete component and a component that is absolutely continuous with respect to the Lebesgue measure) 
  constraint $\e(\tau)=1$, the sequences of coefficients $a_{i,mn}$, $i=0,1,2$ satisfy the following limits:
\begin{eqnarray}
\label{eq:lim a0}
\lim_{n\rightarrow \infty}a_{0,mn}
&=&
\e(\tau^{-1}), 
\\
\label{eq:lim a1}
\lim_{n\rightarrow \infty}a_{1,mn}
&=&
\frac{1}{2},
\\
\label{eq:lim a2}
\lim_{n\rightarrow \infty}n a_{2,mn}^{\rm Clas}
&=&
-\frac{(1-\beta)}{2m}\le 0,
\\
\label{eq:lim a2 SePa}
\lim_{n\rightarrow \infty}n a_{2,mn}^{\rm SePa}
&=&
\left\lbrace
\begin{array}{ll}
0& \mbox{if}\ \tau=1\  \mbox{(Gaussian case)}\\
-\frac{1}{2m}& \mbox{otherwise}\\
\end{array}
\right.
\end{eqnarray}
where
$\beta$ is the probability of the discrete component of the distribution of $\tau$ (i.e., $\beta \pardef \sum_{t;P(\tau=t)\ne 0}P(\tau=t)$.
\end{result}
{\it Proof:}
\eqref{eq:lim a1} follows directly from
\eqref{eq:a0a1} and \eqref{eq:inequalities xi2}.
\eqref{eq:lim a2 SePa} follows from 
$na_{2,mn}^{\rm SePa}=
\frac{n}{\sigma^2_{\mathcal{Q}_{mn}}}- \frac{1}{2m}\xi_{2,mn}$
\eqref{eq:a2clsepa} with \eqref{eq:inequalities xi2} and 
$\frac{n}{\sigma^2_{\mathcal{Q}_{mn}}}=\frac{1}{n[\e(\tau^2)-1]+2\e(\tau^2)}$ derived from \eqref{eq:Stochastic representation CG b}
where $\e(\tau^2)=1$ i.f.f. $\tau=1$ from the Cauchy-Schwarz inequality.
In contrast, the proofs for \eqref{eq:lim a0} and \eqref{eq:lim a2} are significantly more involved and are provided in the Appendix.
\hfill
\QED
%
%%%%%%%%%%%%%%%%%%%%%%%%%%%%%%%%%
%%%%%%%%%%%%%%%%%%%%%%%%%%%%%%%%%%%%%%%
\subsection{Extension of Whittle's formula}
\label{sec:Extension of Whittle's formula}
%%%%%%%%%%%%%%%%%%%%%%%%%%%%%%%%%
%
From the previous two subsections, we can now state our main result:
\begin{result}
\label{main result}
Under the assumptions stated in Section~\ref{sec:Preliminaries on multidimensional stationary SIRP}, the asymptotic FIM rate for a multidimensional stationary SIRP with arbitrary mean, has the following expression:
\begin{eqnarray}
\nonumber
\lim_{n \rightarrow \infty}
\frac{1}{n}\left({\rm FIM}_{y_n}(\boldsymbol{\theta})\right)_{k,\ell}
&=&
\e\left(\tau^{-1}\right) {\boldsymbol \mu}^{'T}_{x,k}
{\bf S}_{x}^{-1}(0)
{\boldsymbol \mu}_{x,\ell}'
\\
\nonumber
&&
\hspace{-2.5cm}
+\frac{1}{2}\int_{0}^{1}\tra({\bf S}_{x}^{-1}(f) {\bf S}_{x,k}^{'}(f) {\bf S}_{x}^{-1}(f){\bf S}_{x,\ell}^{'}(f))df
\\
\nonumber
&&
\hspace{-2.5cm}
{+c}
\left(\int_{0}^{1}\tra({\bf S}_{x}^{-1}(f) {\bf S}_{x,k}^{'}(f))df\right)
\\
\label{eq:Whittle's formula}
&&
\hspace{-1cm}
\left(\int_{0}^{1}\tra({\bf S}_{x}^{-1}(f) {\bf S}_{x,\ell}^{'}(f))df\right),
\end{eqnarray}
where $c$ is given by the limit in \eqref{eq:lim a2} or \eqref{eq:lim a2 SePa}
depending on whether the density generator $g$ is known or unknown, respectively.
\end{result}
Several remarks on Result~\ref{main result} are in order:
\begin{enumerate}
\item {\it Recovery of Whittle's formula.}  
For zero mean stationary GPs, which correspond to the degenerate texture $\tau = 1$, 
we recover, via a novel methodological framework, the classical Whittle formula originally derived in~\cite{Whittle1953} 
or zero mean purely non-deterministic stationary GPs. 
Moreover, our analysis shows that this formula remains valid for any mixture of zero mean Gaussian distributions with proportional covariance matrices, for which the texture r.v. $\tau$ is discrete distributed $(\beta=\gamma=1)$.
This includes, in particular, $\varepsilon$-contaminated Gaussian models.
\item {\it Generalization beyond Gaussianity.}  
For arbitrary zero-mean stationary CGPs whose texture $\tau$ distribution possesses an absolutely continuous component, 
Result~\ref{main result} strictly generalizes Whittle's formula by introducing an additive correcting factor.
This factor reduces the asymptotic FIM rate and depends solely on the probability mass of the distribution of  $\tau$.
\item {\it Extension of Whittle's formula to non-zero mean GPs.}  
For stationary GPs with an arbitrary parameterized mean, Result~\ref{main result} extends Whittle's classical formula by the additive mean-related term ${\boldsymbol \mu}^{'T}_{x,k}
{\bf S}_{x}^{-1}(0){\boldsymbol \mu}_{x,\ell}'$. This aligns with the asymptotic behavior of the maximum likelihood estimator of the mean: the sample mean $\widehat{\boldsymbol \mu}_{x,n}=\frac{1}{n}\sum_{k=1}^n {\bf x}_k$ satisfies~\cite[Th.~7.1.1]{Brockwell1990}
 $\lim_{n\rightarrow \infty}n\cov(\widehat{\boldsymbol \mu}_{x,n})=\sum_k{\bf R}_x(k)={\bf S}_x(0)$.
\item {\it Enhanced information in compound Gaussian models.}  
For stationary CGPs with parameterized mean, the ML estimator of the mean is generally not the sample mean. 
The coefficient of the mean-related term becomes $\e(1/\tau)$. Under the normalization $\e(\tau) = 1$, Jensen's inequality yields
$\e\!\left(\tau^{-1}\right) \geq \frac{1}{\e(\tau)} = 1$,
with equality i.f.f. $\tau = 1$  (i.e., the Gaussian case). Hence, the asymptotic FIM for CGPs is strictly larger than for GPs when the texture is non-degenerate, reflecting greater information about the mean due to the hierarchical structure.
\item {\it Whittle's formula for i.i.d.\ CG r.v.}
For i.i.d.\ CG r.v. $({\bf x}_k)_{k=1}^n$
with common mean $\boldsymbol{\mu}_x$ and covariance
${\bf R}_{x}$, the exact FIM rate obtained from the
Slepian-Bangs formula~\eqref{eq:Slepian-Bangs's formula}
is given by
\begin{eqnarray}
\nonumber
\frac{1}{n}\left({\rm FIM}_{{\bf y}_n}(\boldsymbol{\theta})\right)_{k,\ell}
&=&
a_{0,m}\, {\boldsymbol \mu}^{\prime T}_{x,k}
{\bf R}_{x}^{-1}
{\boldsymbol \mu}_{x,\ell}'
\\
\nonumber
&&
\hspace{-1.7cm}
+a_{1,m}\,\mathrm{tr}({\bf R}_{x}^{-1} {\bf R}_{x,k}^{'}
{\bf R}_{x}^{-1}{\bf R}_{x,\ell}^{'})
\\
\label{eq:Whittle iid ES}
&&
\hspace{-1.7cm}
+a_{2,m}\,\mathrm{tr}({\bf R}_{x}^{-1}{\bf R}_{x,k}^{'})
\,\mathrm{tr}({\bf R}_{x}^{-1} {\bf R}_{x,\ell}^{'}) .
\end{eqnarray}
However, comparing \eqref{eq:Whittle iid ES} with the
asymptotic FIM rate in~\eqref{eq:Whittle's formula}, where
the spectrum of ${\bf x}_k$ is flat, i.e.,
${\bf S}_{x}(f)={\bf R}_{x}$, reveals an apparent
inconsistency. This inconsistency is resolved by noting
that the stacked vector
${\bf y}_n = ({\bf x}_1^T, \dots, {\bf x}_n^T)^T$
is not CG unless the underlying distribution is Gaussian,
in which case both expressions coincide.
\item {\it A specific parameterization of
$({\boldsymbol \mu}_x, ({\bf R}_{x}(k))_{k \in \mathbb{Z}})$.}
Consider the particular case in which ${\boldsymbol \mu}_x$
and $({\bf R}_{x}(k))_{k \in \mathbb{Z}}$ share no common
parameter. Specifically, let ${\boldsymbol \mu}_x$ be
parameterized by $\boldsymbol{\theta}_1$, and let
${\bf R}_{x}(k) = \theta_2\,{\bf R}(k,\boldsymbol{\theta}_3)$,
where $\theta_2$ is a power factor and
$\boldsymbol{\theta}_3$ is a shape factor.
In this case, the Slepian--Bangs
formula~\eqref{eq:Slepian-Bangs's formula} immediately
implies that $\boldsymbol{\theta}_1$ and
$(\theta_2,\boldsymbol{\theta}_3)$ are decoupled in
${\rm FIM}_{{\bf y}_n}$. Moreover,
${\rm CRB}_{{\bf y}_n}(\boldsymbol{\theta}_1)$,
${\rm CRB}_{{\bf y}_n}(\theta_2)$, and
${\rm CRB}_{{\bf y}_n}(\boldsymbol{\theta}_3)$ are
proportional to $\theta_2$, proportional to $\theta_2^2$,
and independent of $\theta_2$, respectively.
In particular, for an arbitrary distribution of the texture
$\tau$, we have
\begin{equation}
\label{CRB theta1}
\lim_{n \rightarrow \infty}\, n\,{\rm CRB}_{{\bf
y}_n}(\boldsymbol{\theta}_1)
=
\frac{\theta_2}{\e\!\left(\tau^{-1}\right)}
\!\left(\frac{\partial {\boldsymbol \mu}_{x}^T}
{\partial {\boldsymbol \theta}_1^T}
{\bf S}^{-1}(0,\boldsymbol{\theta}_3)
\frac{\partial {\boldsymbol \mu}_{x}}
{\partial {\boldsymbol \theta}_1^T}\!\right)^{-1},
\end{equation}
where ${\bf S}(f,\boldsymbol{\theta}_3) \pardef
\sum_k{\bf R}(k,\boldsymbol{\theta}_3)\,e^{-i2\pi kf}$.
Furthermore, it is straightforward to show that
\begin{eqnarray}
\nonumber
&&
\hspace{-1cm}
\int_{0}^{1}\tra\!\left({\bf S}_{x}^{-1}(f)
{\bf S}_{x,k}^{'}(f) {\bf S}_{x}^{-1}(f)
{\bf S}_{x,\ell}^{'}(f)\right)df
\\
\nonumber
&&
=\frac{1}{m}
\left(\int_{0}^{1}\tra\!\left({\bf S}_{x}^{-1}(f)
{\bf S}_{x,k}^{'}(f)\right)df\right)
\\
\label{eq:Whittle's formula b}
&&
\hspace{1cm}
\times\left(\int_{0}^{1}\tra\!\left({\bf S}_{x}^{-1}(f)
{\bf S}_{x,\ell}^{'}(f)\right)df\right),
\end{eqnarray}
when $k$ indexes the parameter $\theta_2$.
Consequently, applying \eqref{eq:Whittle's formula} to a
purely continuous ($\beta=\gamma=0$) distribution of the
texture $\tau$, the asymptotic FIM
rate~\eqref{eq:Whittle's formula} exhibits the block
structure
$\begin{bmatrix}
\times & {\bf 0} & {\bf 0}\\
{\bf 0}^T & 0 & {\bf 0}^T\\
{\bf 0} & {\bf 0} & \times
\end{bmatrix}$,
and hence no general conclusion can be drawn regarding the
asymptotic behavior of
${\rm CRB}_{{\bf y}_n}({\theta}_2)$ and
${\rm CRB}_{{\bf y}_n}(\boldsymbol{\theta}_3)$ as
$n \to \infty$.
An example illustrating this behavior is provided in
Section~\ref{sec:Numerical illustrations}.
\item  {\it Practical accuracy of the asymptotic CRB approximation.}  
The numerical accuracy of the CRB derived from the asymptotic expression~\eqref{eq:Whittle's formula} depends not only on the sample size $n$, but also on the spectral characteristics of the process  ${\bf x}_k$ and the distribution of $\tau$. In practice, however, the CRB values obtained from~\eqref{eq:Whittle's formula} closely match those computed using the Slepian-Bangs formula~\eqref{eq:Slepian-Bangs's formula}, even for moderate or small $n$, particularly when the spectrum is sharply peaked (i.e., highly correlated processes). This empirical agreement will be illustrated in Section~\ref{sec:Numerical illustrations}.
\end{enumerate}
%
%%%%%%%%%%%%%%%%%%%%%%%%%%%%%%%%%%%%%%%%%%%%%%%%%%%%%%%%
\section{Numerical Illustrations}
\label{sec:Numerical illustrations}
%%%%%%%%%%%%%%%%%%%%%%%%%%%%%%%%%%%%%%%%%
%
To illustrate Result~\ref{main result}, we consider a scalar
nonzero-mean AR(1) process $(x_k)_{k \in \mathbb{Z}}$, for which
$R_x(k) = \sigma_x^2\, a^{|k|}$
and
$S_x(f) =
\frac{\sigma_x^2(1-a^2)}
{\left|1 - a\,e^{-i2\pi f}\right|^2}$,
with $|a| < 1$.
The unknown parameter vector, which satisfies the condition
stated in Remark~6 of Result~\ref{main result}, is
$\boldsymbol{\theta} = (\mu_x, \sigma_x^2, a)^T$.

Substituting a closed-form expression for
$\boldsymbol{\Sigma}_{{\bf y}_n}^{-1}$ into the
Slepian--Bangs formula~\eqref{eq:Slepian-Bangs's formula},
we obtain
\begin{equation}
\label{eq;FIM mu}
{\rm FIM}_{{\bf y}_n}(\mu_x)
=n\,\frac{a_{0,n}}{\sigma_x^2}
\left(\frac{1-a}{1+a}
+ \frac{2a}{n(1+a)}\right),
\end{equation}
and, using symbolic computation,
\begin{eqnarray}
\nonumber
{\rm FIM}_{{\bf y}_n}(\sigma_x^2,a)=
&&
\\
\label{eq:FIM sigma a}
&&
\hspace{-3.7cm}
\begin{bmatrix}
\!\frac{n}{\sigma^4_{x}}(a_{1,n}+n a_{2,n}) & -\frac{2a(n-1)}{\sigma^2_x(1-a^2)}(a_{1,n}+n a_{2,n})\!\\
\!-\frac{2a(n-1)}{\sigma^2_x(1-a^2)}(a_{1,n}\!+\! n a_{2,n}) & a_{1,n}\frac{2(n\!-\!1)(1+a^2)}{(1-a^2)^2}\!+\!a_{2,n}\frac{4a^2(n\!-\!1)^2}{(1-a^2)^2}\!
\end{bmatrix}.
\end{eqnarray}

As a preliminary consistency check, noting that
$S_x(0) = \sigma_x^2\!\left(\frac{1+a}{1-a}\right)$,
we see that \eqref{eq;FIM mu} is consistent
with~\eqref{CRB theta1}
and gives: 
\begin{equation}
\label{eq:CRB mu CG}
\lim_{n \rightarrow \infty}
n\,{\rm CRB}_{{\bf y}_n}(\mu_x)
=
\frac{\sigma_x^2}{\e(\tau^{-1})}
\left(\frac{1+a}{1-a}\right).
\end{equation}
Furthermore, \eqref{eq:FIM sigma a} yields:
\begin{eqnarray}
\nonumber
{\rm CRB}_{{\bf y}_n}(\sigma_x^2)
&&
\\
\label{eq:CRB sigma}
&&
\hspace{-2cm}
=\frac{\sigma_x^4}{n}
\left[
\frac{(1+a^2)\,a_{1,n}
+ 2a^2\frac{(n-1)}{n}\,n\,a_{2,n}}
{a_{1,n}(a_{1,n}+n\,a_{2,n})
\left(1-a^2+\frac{2a^2}{n}\right)}
\right],
\\[6pt]
\label{eq:CRB a}
{\rm CRB}_{{\bf y}_n}(a)
&=&
\frac{(1-a^2)^2}
{2(n-1)\,a_{1,n}
\left(1-a^2+\frac{2a^2}{n}\right)}.
\end{eqnarray}

Using the asymptotic properties of the sequence $a_{1,n}$ established in
Result~\ref{limit of the three coefficients}, we obtain, for
an arbitrary distribution of the texture $\tau$,
\begin{equation}
\label{eq;CRB a bis}
\lim_{n \rightarrow \infty}
n\,{\rm CRB}_{{\bf y}_n}(a)
= 1-a^2.
\end{equation}

In contrast, the asymptotic behavior of
${\rm CRB}_{{\bf y}_n}(\sigma_x^2)$ depends on the
distribution of $\tau$, and must be analyzed
separately for discrete and continuous components.

When this distribution has a discrete component
($\beta \neq 0$), \eqref{eq:CRB sigma} yields
\begin{equation}
\label{eq:CRB sigma b}
\lim_{n \rightarrow \infty}
n\,{\rm CRB}_{{\bf y}_n}^{\rm Clas}(\sigma_x^2)
=
\frac{2\sigma_x^4[1-a^2(1-2\beta)]}
{(1-a^2)\,\beta}.
\end{equation}
In particular, for Gaussian processes
$(x_k)_{k \in \mathbb{Z}}$ (i.e., $\tau=1$ and thus $\beta = 1$), \eqref{eq:CRB sigma b}
reduces to
\begin{equation}
\lim_{n \rightarrow \infty}
n{\rm CRB}_{{\bf y}_n}^{\rm Clas}(\sigma_x^2)
=
\lim_{n \rightarrow \infty}
n{\rm CRB}_{{\bf y}_n}^{\rm SePa}(\sigma_x^2)
=
\frac{2\sigma_x^4(1+a^2)}{1-a^2},
\end{equation}
confirming that the classical and semiparametric bounds coincide in the Gaussian case.

When the distribution of $\tau$ is purely continuous (for
known $g$), or is not reduced to the degenerate case
$\tau = 1$ (for unknown $g$), Result~\ref{limit of the three coefficients}  yields
$\lim_{n \rightarrow \infty}(a_{1,n} + n\,a_{2,n}) = 0$. Consequently from \eqref{eq:CRB sigma}, the asymptotic behavior of ${\rm CRB}_{{\bf y}_n}(\sigma_x^2)$ cannot be determined from the limiting FIM alone and must be examined on a case-by-case basis, depending on the rate at which $a_{1,n} + n\,a_{2,n}$ converges to zero.
 As an illustrative example, consider a Student's-$t$ process
$(x_k)_{k \in \mathbb{Z}}$ with $\nu > 2$ degrees of
freedom.  In this case, 
$n(a_{1,n} + n\,a_{2,n})
= \frac{n\,\nu}{2(n+\nu+2)}$,
and thus \eqref{eq:CRB sigma} yields
\begin{equation}
\label{eq:CRB sigma c}
\lim_{n \rightarrow \infty}
{\rm CRB}_{{\bf y}_n}^{\rm Clas}(\sigma_x^2)
=
\lim_{n \rightarrow \infty}
{\rm CRB}_{{\bf y}_n}^{\rm SePa}(\sigma_x^2)
=
\frac{2\sigma_x^4}{\nu}.
\end{equation}

In the following, all the figures concern a Student's-$t$
process.
Figs.~1 and~2 compare the exact value
$n\!\left(a_{0,n}\,\mathds{1}_{n}^T 
\boldsymbol{\Sigma}_{{\bf y}_n}^{-1}
\mathds{1}_{n}\right)^{-1}$
of $n\,{\rm CRB}_{{\bf y}_n}(\mu_x)$ with the asymptotic
limit \eqref{eq:CRB mu CG} evaluated with $\e(\tau^{-1})=\frac{\nu}{\nu-2}$
as a function of $n$, for different values of $\nu$ and $a$, respectively.
These figures show that the asymptotic limits are reached
rapidly, typically for sample sizes smaller than $100$. In Fig.~1, the limit in  \eqref{eq:CRB mu CG} is reached more rapidly as the
distribution moves away from the Gaussian case ($\nu \gg 1$). Fig.~2 shows that this limit, which is an increasing function
of $a\in(-1,1)$, is approached less rapidly as $a$ tends to $1$
(long-range positive correlation); for $a=0.95$, convergence is
still not complete at $n=10^3$.
%
%%%%%%%%%%%%%%%%%%%%%%%%%%%%%%%%%%%%%%%%%%%%%
\begin{figure}[htbp]
\centering
\includegraphics[width=0.4\textwidth]{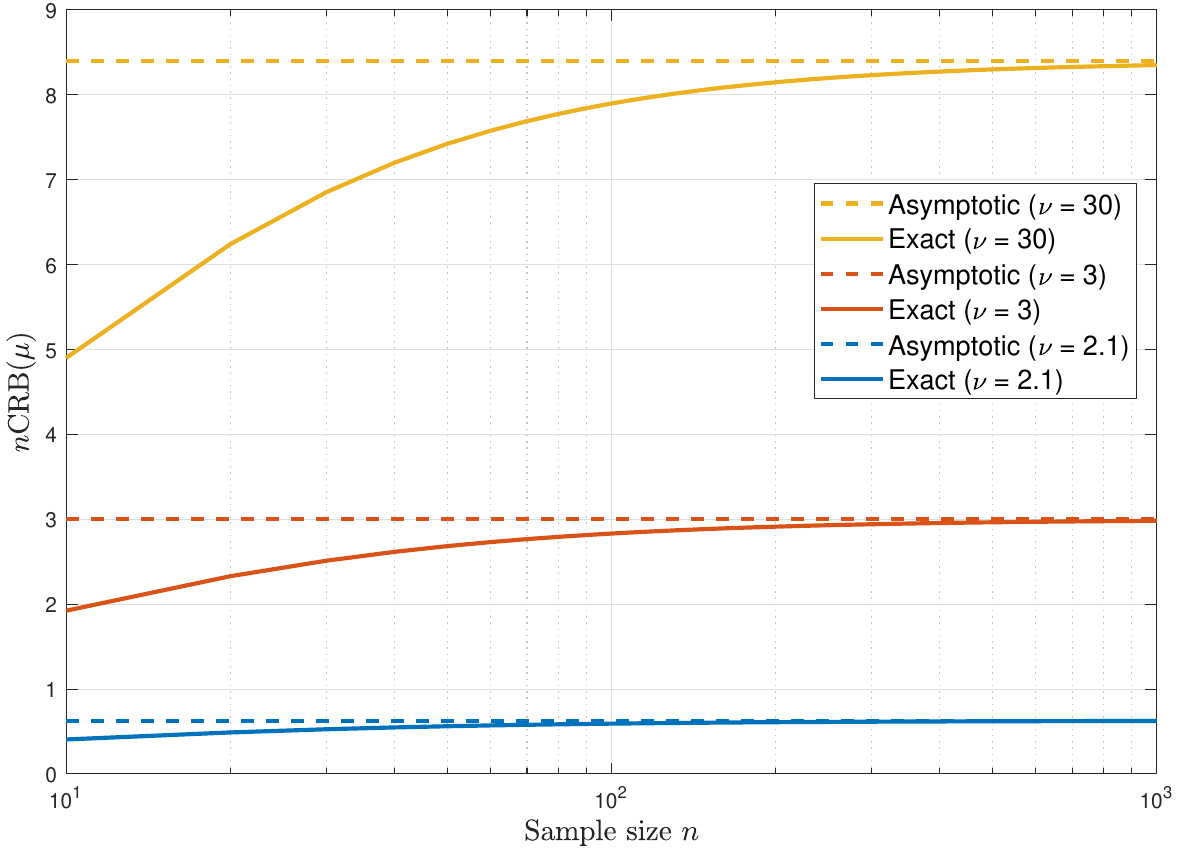}
\caption{$n{\rm CRB}_{y_n}(\mu_x)$ as a function of $n$ for different values of $\nu$, with $a=0.8$ and $\sigma_x^2=1$.}
\label{fig:Fig1}
\end{figure}
 \begin{figure}[htbp]
\centering
\includegraphics[width=0.4\textwidth]{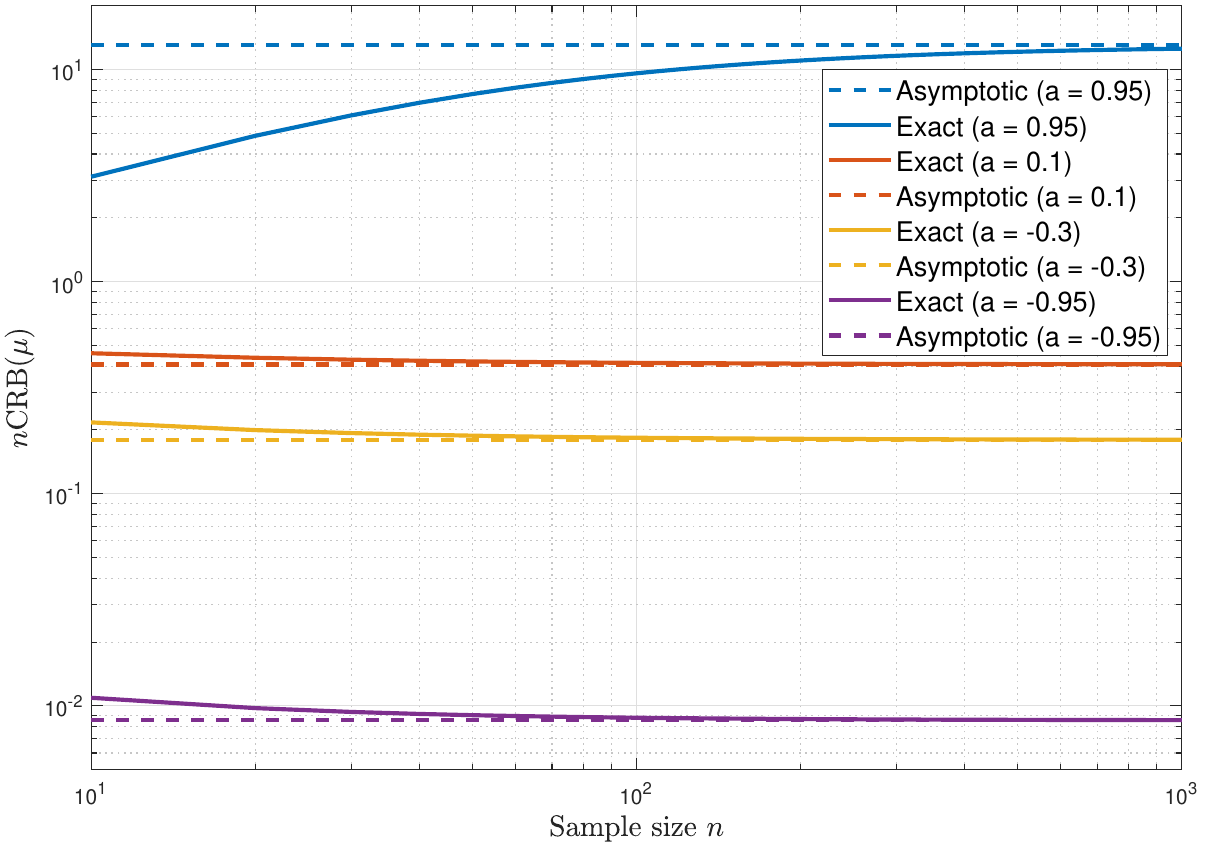}
\caption{$n{\rm CRB}_{y_n}(\mu_x)$ as a function of $n$ for different values of $a$, with $\nu=3$ and $\sigma_x^2=1$.}
\label{fig:Fig2}
\end{figure}
Finally, Figs.~3 and~4 show the  behavior of ${\rm CRB}_{y_n}(\sigma_x^2)$ as a  function of $n$ for different value of $\nu$ and $a$, respectively.

\begin{figure}[htbp]
\centering
\includegraphics[width=0.4\textwidth]{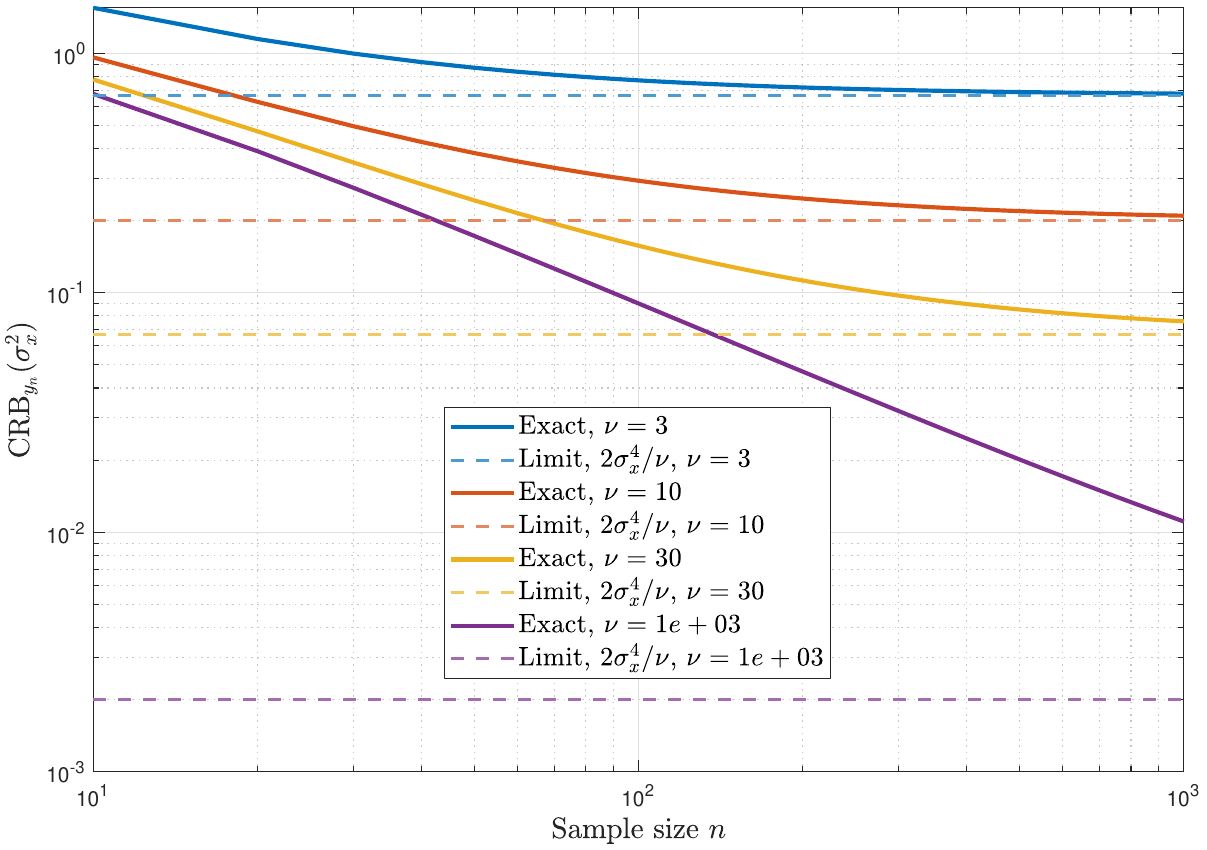}
\caption{${\rm CRB}_{y_n}(\sigma_x^2)$ as a function of $n$ for different values of $\nu$, with $a=0.8$ and $\sigma_x^2=1$.}
\label{fig:Fig3}
\end{figure}

In Fig.~3, the closer the distribution of
$(x_k)_{k\in\mathbb{Z}}$ is to Gaussian, the lower the floor of
${\rm CRB}_{{\bf y}_n}(\sigma_x^2)$; as $\nu$ increases this
floor decreases and, in the Gaussian limit, vanishes altogether,
leaving the standard $1/n$ decay.
This provides another example in which the Gaussian assumption
does not yield the largest CRB \cite{Delmas2026}.
In Fig.~4, ${\rm CRB}_{{\bf y}_n}(\sigma_x^2)$ depends on $a$
only through $a^2$ (hence not on its sign), and converges to the
$a$-independent limit \eqref{eq:CRB sigma c}, which is reached
more quickly as $a$ approaches zero. Fig.~3 shows that the closer the distribution of
$(x_k)_{k\in\mathbb{Z}}$ is to Gaussian (i.e., the larger
$\nu$), the lower the floor $2\sigma_x^4/\nu$ of
${\rm CRB}_{{\bf y}_n}(\sigma_x^2)$; in the Gaussian limit this
floor vanishes altogether, and the standard $1/n$ decay is
recovered.
 \begin{figure}[htbp]
\centering
\includegraphics[width=0.4\textwidth]{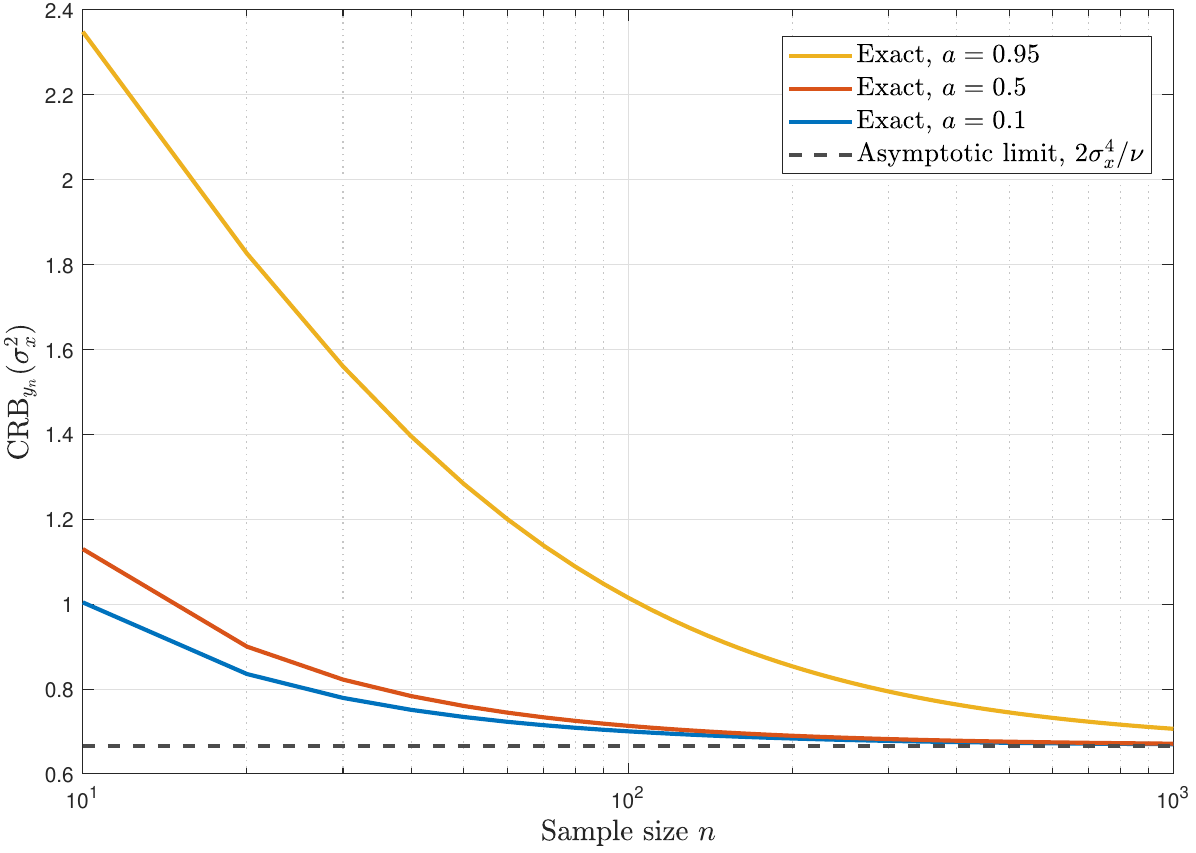}
\caption{${\rm CRB}_{y_n}(\sigma_x^2)$ as a function of $n$ for different values of $a$, with $\nu=3$ and $\sigma_x^2=1$.}
\label{fig:Fig4}
\end{figure}
Fig.~4 further shows that ${\rm CRB}_{{\bf y}_n}(\sigma_x^2)$ converges to the
$a$-independent limit \eqref{eq:CRB sigma c}, which is reached
more quickly as $a$ approaches zero.
%
%%%%%%%%%%%%%%%%%%%%%%%%%%%%%%%%%%%%%%%%%%%%%%%%%%%%%%
%%%%%%%%%%%%%%%%%%%%%%%%%%%%%%%%%%%%%%%%%%%
\section{Conclusion}
\label{sec:Conclusion}
%%%%%%%%%%%%%%%%%%%%%%%%%%
%
This paper has established the asymptotic limit of the FIM rate for multidimensional, real-valued, stationary compound Gaussian processes (CGPs) with arbitrary mean.  The result constitutes a generalization of Whittle's formula, derived by combining the Slepian-Bangs expression with asymptotic theory for block Toeplitz matrices. Importantly, the proposed framework unifies three statistical settings, fully known, parameterized, and completely unknown CG-dependent density generators, while providing explicit analytical expressions for all correction terms with respect to Whittle's formula. In addition, a semiparametric formulation is developed for CG distributions, extending the general semiparametric FIM framework to the CG setting and quantifying the performance loss due to the unknown distribution.
Current research efforts are focused on extending this asymptotic framework to complex-valued stationary and cyclostationary multivariate CGPs, as well as on applying the derived FIM to practical parameter estimation problems, particularly direction-of-arrival and impulse-response estimation.
%
%%%%%%%%%%%%%%%%%%%%%%%%%%%%%%%%%%%%%%%%%%%%%%%
%\vspace{-5cm}
%%%%%%%%%%%%%%%%%%%%%%%%%%%%%%%%%%%%%%%%%%%%%%%%%%%%%%%%%%%%ùù
\section*{Appendix}
\label{sec:Appendix}
%%%%%%%%%%%%%%%%%%%%%%%%%%%%%%%%%%%%%%%%%%%%%%%%%%
%%%%%%%%%%%%%%%%%%%%%%%%%%%%%%%%%%%%%%%%%%%%%%%%%%
%%%%%%%%%%%%%%%%%%%%%%%%%%%%%%%%%%%%%%%%%%%%%%%%%%
\subsection{Proof of Result \ref{property FIM inequality}}
\label{Proof of Result property FIM inequality}
%%%%%%%%%%%%%%%%%%%%%%%%%%%%%%%%%%%%%%%%%ù
%
From \eqref{eq:a2class} and \eqref{eq:a2clpa}, the inequality \eqref{eq:inequality a2 clas sepa} is equivalent to 
$\xi_{2,mn}\ge \frac{mn}{mn+2}\left(1+\frac{4}{\sigma^2_Q}\right)$ and from \eqref{eq:E(Qphi)2}
and \eqref{eq:e e}, we get:
\begin{equation}
\label{eq:xi2 class separ}
\xi_{2,mn}
=
\frac{mn}{mn+2}\left(1+\frac{\var(\e[\chi^2_{mn}|\mathcal{Q}_{mn}])}{(mn)^2}\right).
\end{equation}
Then from the Cauchy-Schwarz inequality, we obtain:
\begin{equation}
\label{eq:Cauchy-Schwarz 1}
\{\cov(\e[\chi^2_{mn}|\mathcal{Q}_{mn}],\!\mathcal{Q}_{mn})\}^2
\!\!\le\!\!
\var(\e[\chi^2_{mn}|\mathcal{Q}_{mn}])
\var(\mathcal{Q}_{mn}).
\end{equation}
Applying the identity\footnote{derived  from $\cov(\e[X|Y],Y)=\e(Y\e[X|Y])-\e(\e[X|Y])\e(Y)$
where $\e(Y\e[X|Y])=\e(\e[XY|Y])=\e(XY)$ and $\e(\e[X|Y])=\e(X)$.} 
$\cov(\e[X|Y],Y)=\cov(X,Y)$ 
with $X = \chi^2_{mn}$ and $Y = \mathcal{Q}_{mn}$, we have
\begin{eqnarray}
\nonumber
\cov(\e[\chi^2_{mn}|\mathcal{Q}_{mn}],\mathcal{Q}_{mn})
&=&
\cov(\chi^2_{mn},\mathcal{Q}_{mn})
\label{eq:cov}
\\
&&
\hspace{-4cm}
=
\e(\tau)\e([\chi^2_{mn}]^2)-\e(\chi^2_{mn})\e(\mathcal{Q}_{mn})
=2mn.
\end{eqnarray}
Consequently from \eqref{eq:Cauchy-Schwarz 1}, $\var(\e[\chi^2_{mn}|\mathcal{Q}_{mn}])\ge \frac{4(mn)^2}{\sigma_{\mathcal{Q}}^2}$ with equality i.f.f.
the r.v. $\e[\chi^2_{mn}|\mathcal{Q}_{mn}]-mn=\mathcal{Q}_{mn}\e(\tau^{-1}|\mathcal{Q}_{mn}]-mn$ and $\mathcal{Q}_{mn}-m$ are proportional, i.e., 
$\e(\tau^{-1}|\mathcal{Q}_{mn}]=1\Leftrightarrow\varphi(\mathcal{Q}_{mn})=1$.
The latter holds i.f.f. $g(t)=1/ (2\pi)^{mn/2}\exp(-t/2)$
which characterizes the Gaussian distribution. This completes the proof.
\hfill
\QED
%
%%%%%%%%%%%%%%%%%%%%%%%%%%%%%%%%%%%%%%%%%%%%%%%%%%
\subsection{Proof of Result \ref{Result matrice mean limit}}
\label{Proof of Result Result matrice mean limit}
%%%%%%%%%%%%%%%%%%%%%%%%%%%%%%%%%%%%%%%%%ù
%
Since the rules of matrix multiplication carry over to block-matrices, for each $i=1,\dots,n$ we have:
\begin{equation}
\label{eq:mean1}
[\boldsymbol{\Sigma}_{{\bf y}_n}
(\mathds{1}_{n}\otimes {\bf I}_{m})]_{[i,1]} =\sum_{k=-i+1}^{n-i}{\bf R}_x(k).
\end{equation}

%\HAT{Message pour Habti : attention ici $(k,\ell)$-th block est donné par  ${\bf R}_{x}(\ell-k)$ car 
%${\bf R}_{x}(k) \pardef \e[({\bf x}_{\ell}-{\boldsymbol \mu}_x)({\bf x}_{\ell+k}-{\boldsymbol \mu}_x)^T] \in \mathbb{R}^{m \times m}$ et donc la première ligne est  $({\bf R}_x(0),{\bf R}_x(1),..,{\bf R}_x(n\!-\!1))$, j'ai donc rétabli mes formules.}

\noindent
Comparing this matrix to the matrix
$[(\mathds{1}_{n}\otimes {\bf I}_{m}){\bf S}_x(0)]_{[i,1]}=[(\mathds{1}_{n}\otimes {\bf I}_{m})\sum_{k}{\bf R}_x(k)]_{[i,1]}=\sum_{k}{\bf R}_x(k)$, we derive:
\begin{eqnarray}
\nonumber
[\boldsymbol{\Sigma}_{{\bf y}_n}
(\mathds{1}_{n}\otimes {\bf I}_{m})
-(\mathds{1}_{n}\otimes {\bf I}_{m}){\bf S}_x(0)]_{[i,1]}
\\
\label{eq:mean2}
&&
\hspace{-5.7cm}
=
\sum_{k<-i+1}{\bf R}_x(k)+\sum_{k>n-i}{\bf R}_x(k)
\pardef [{\boldsymbol \Upsilon}_{n}]_{[i,1]}.
\end{eqnarray}
Using the identity   $\|{\bf A}\|_2^2=\|{\bf A}^T{\bf A}\|_2)$ and the definition of the spectral norm, we obtain
\begin{eqnarray}
\nonumber
\|\boldsymbol{\Sigma}_{{\bf y}_n} \frac{1}{\sqrt n}(\mathds{1}_{n}\otimes {\bf I}_{m})
- \frac{1}{\sqrt n}(\mathds{1}_{n}\otimes {\bf I}_{m}){\bf S}_x(0)\|_2^2
&&
\\
\label{eq:mean3}
&&
\hspace{-3.2cm}
\le \frac{1}{n} \sum_{i=1}^n \|[{\boldsymbol \Upsilon}_{n}]_{[i,1]}\|_2^2.
\end{eqnarray}
From \eqref{eq:mean2} and the triangle inequality,
\begin{equation}
\label{eq:mean4}
\|[{\boldsymbol \Upsilon}_{n}]_{[i,1]}\|_2
\le
\sum_{k<-i+1}\|{\bf R}_x(k)\|_2+\sum_{k>n-i}\|{\bf R}_x(k)\|_2.
\end{equation}
Then from the absolute summability of the sequence ${\bf R}_x(k)$, the sequence 
$\alpha_n \pardef  \sum_{k<-n}\|{\bf R}_x(k)\|_2)\|_2$ and 
$\beta_n \pardef \sum_{k>n}\|{\bf R}_x(k$ tend to $0$ when $n \rightarrow \infty$ and \eqref{eq:mean4} can be written equivalently as
$\|[{\boldsymbol \Upsilon}_{n}]_{[i,1]}\|_2 \le\alpha_{i-1}+ \beta_{n-i}$, 
which implies $\|[{\boldsymbol \Upsilon}_{n}]_{[i,1]}\|_2^2
\le 2\alpha_{i-1}^2+2 \beta_{n-i}^2$ and therefore:
\begin{equation}
\label{eq:mean5}
\frac{1}{n}\sum_{i=1}^n\|[{\boldsymbol \Upsilon}_{n}]_{[i,1]}\|_2^2
\le
\frac{2}{n}\sum_{i=0}^{n-1}\alpha_i^2
+\frac{2}{n}\sum_{i=0}^{n-1}\beta_i^2.
\end{equation}
Now applying the Ces\`{a}ro mean theorem to the sequences $\alpha_n^2$ and $\beta_n^2$, we have 
$\lim_{n \rightarrow \infty}\frac{1}{n}\sum_{i=0}^{n-1}\alpha_i^2=0$ and
$\lim_{n \rightarrow \infty}\frac{1}{n}\sum_{i=0}^{n-1}\beta_i^2=0$, and therefore from \eqref{eq:mean3}
we get:
\begin{equation}
\label{eq:mean6}
\lim_{n \rightarrow \infty}\|\boldsymbol{\Sigma}_{{\bf y}_n} \frac{1}{\sqrt n}(\mathds{1}_{n}\otimes {\bf I}_{m})
-\! \frac{1}{\sqrt n}(\mathds{1}_{n}\otimes {\bf I}_{m}){\bf S}_x(0)\|_2
\!=\!0,
\end{equation}
or in an equivalent way $\lim_{n \rightarrow \infty}\frac{1}{\sqrt{n}}\|{\boldsymbol \Upsilon}_{n}\|_2=0$ where
${\boldsymbol \Upsilon}_{n}\pardef -\left[\boldsymbol{\Sigma}_{{\bf y}_n} (\mathds{1}_{n}\otimes {\bf I}_{m})
- (\mathds{1}_{n}\otimes {\bf I}_{m}){\bf S}_x(0)\right]$.
This implies:
\begin{eqnarray}
\nonumber
\frac{1}{n}(\mathds{1}_{n}^T\otimes {\bf I}_{m})
\boldsymbol{\Sigma}_{{\bf y}_n}^{-1}
(\mathds{1}_{n}\otimes {\bf I}_{m})
&=&
{\bf S}_x^{-1}(0)
\\
\label{eq:mean7}
&&
\hspace{-2.5cm}
+\frac{1}{n}(\mathds{1}_{n}^T\otimes {\bf I}_{m})\boldsymbol{\Sigma}_{{\bf y}_n}^{-1}{\boldsymbol \Upsilon}_{n} {\bf S}_x^{-1}(0).
\end{eqnarray}
Taking norms and using $\|\mathds{1}_{n}^T\otimes {\bf I}_{m}\|_2=\sqrt{n}$, we obtain:
\begin{eqnarray}
\nonumber
\|\frac{1}{n}(\mathds{1}_{n}^T\otimes {\bf I}_{m})
\boldsymbol{\Sigma}_{{\bf y}_n}^{-1}
(\mathds{1}_{n}\otimes {\bf I}_{m})
-{\bf S}_x^{-1}(0)\|_2
&&
\\
\label{eq:mean8}
&&
\hspace{-4.5cm}
\le \frac{1}{\sqrt n}\|\boldsymbol{\Sigma}_{{\bf y}_n}^{-1}\|_2\|{\boldsymbol \Upsilon}_{n}\|_2 \| {\bf S}_x^{-1}(0)\|_2.
\end{eqnarray}
Using $\|\boldsymbol{\Sigma}_{{\bf y}_n}^{-1}\|_2 <c_r$
from the assumptions given in Section \ref{sec:Preliminaries on multidimensional stationary SIRP}, \eqref{eq:mean8} concludes the proof.
\hfill
\QED
%
%%%%%%%%%%%%%%%%%%%%%%%%%%%%%%%%%%%%%%%%%%%%%%%%%%
\subsection{Proof of Result \ref{Result matrice variance limit}}
\label{Proof of Result Result matrice variance limit}
%%%%%%%%%%%%%%%%%%%%%%%%%%%%%%%%%%%%%%%
%
From the property 
 ${\rm Inf}_{f \in [0,1)} \lambda_{\min}({\bf S}_{x}(f))>0$,
\cite[Th. 6.4]{Gutierrez2011} implies that $[{\bf T}_n({\bf S}_{x}(f))]^{-1} \sim {\bf T}_n({\bf S}_{x}^{-1}(f))$.
Then it follows from \cite[Lemma 3.2]{Gutierrez2011} that
\begin{equation}
\label{eq:Gray1}
[{\bf T}_n({\bf S}_{x}(f))]^{-1}{\bf T}_n( {\bf S}_{x,k}^{'}(f))
\sim
{\bf T}_n({\bf S}_{x}^{-1}(f)){\bf T}_n ({\bf S}_{x,k}^{'}(f)).
\end{equation}
Furthermore, it follows  from \cite[Th. 6.2]{Gutierrez2011} that
\begin{equation}
\label{eq:Gray2}
{\bf T}_n({\bf S}_{x}^{-1}(f)){\bf T}_n ({\bf S}_{x,k}^{'}(f))
\sim
{\bf T}_n({\bf S}_{x}^{-1}(f){\bf S}_{x,k}^{'}(f)),
\end{equation}
and from \cite[Lemma 6.1]{Gutierrez2011}
\begin{equation}
\label{eq:Gray3}
{\bf T}_n({\bf S}_{x}^{-1}(f){\bf S}_{x,k}^{'}(f))
\sim
{\bf C}_n({\bf S}_{x}^{-1}(f){\bf S}_{x,k}^{'}(f)).
\end{equation}
This implies because $\sim$ is an equivalence relation that
\begin{equation}
\label{eq:Gray4}
[{\bf T}_n({\bf S}_{x}(f))]^{-1}{\bf T}_n( {\bf S}_{x,k}^{'}(f))
\sim
{\bf C}_n({\bf S}_{x}^{-1}(f){\bf S}_{x,k}^{'}(f)).
\end{equation}
According to the definition of asymptotically equivalent sequences \cite[Definition 3.1]{Gutierrez2011}, \eqref{eq:Gray4} formally expressed as:
\begin{eqnarray}
\nonumber
&&\lim_{n\rightarrow \infty}\|[{\bf T}_n({\bf S}_{x}(f))]^{-1}{\bf T}_n( {\bf S}_{x,k}^{'}(f))
\\
\label{eq:Gray5}
&&
\hspace{1.5cm}
-{\bf C}_n({\bf S}_{x}^{-1}(f){\bf S}_{x,k}^{'}(f))\|_F /\sqrt{n}=0.
\end{eqnarray}
Then, applying \cite[Proposition 2.2]{Gutierrez2011}, we get:
\begin{eqnarray*}
0
&\leq&
\left|\frac{1}{n}\tra\{[{\bf T}_n({\bf S}_{x}(f))]^{-1}{\bf T}_n( {\bf S}_{x,k}^{'}(f))\}\right.
\\
&&
\hspace{2.5cm}
\left.-\frac{1}{n}\tra\{{{\bf C}_n({\bf S}_{x}^{-1}(f){\bf S}_{x,k}^{'}(f))\}}\right|
\\
&=&
|\tra\{[{\bf T}_n({\bf S}_{x}(f))]^{-1}{\bf T}_n( {\bf S}_{x,k}^{'}(f))
\\
&&
\hspace{2.5cm}
-{\bf C}_n({\bf S}_{x}^{-1}(f){\bf S}_{x,k}^{'}(f))\}| / n
\\
&\le&
\sqrt{mn} \|[{\bf T}_n({\bf S}_{x}(f))]^{-1}{\bf T}_n( {\bf S}_{x,k}^{'}(f))
\\
&&
\hspace{2.5cm}
-{\bf C}_n({\bf S}_{x}^{-1}(f){\bf S}_{x,k}^{'}(f))\|_F / n
\\
&=&
\sqrt{m} \|[{\bf T}_n({\bf S}_{x}(f))]^{-1}{\bf T}_n( {\bf S}_{x,k}^{'}(f))
\\
&&
\hspace{2.5cm}
-{\bf C}_n({\bf S}_{x}^{-1}(f){\bf S}_{x,k}^{'}(f))\|_F / \sqrt{n},
\end{eqnarray*}
Therefore
\begin{eqnarray}
\label{eq:Gray6}
&&\lim_{n\rightarrow \infty}
\left|\frac{1}{n}\tra\{[{\bf T}_n({\bf S}_{x}(f))]^{-1}{\bf T}_n( {\bf S}_{x,k}^{'}(f))\}\right.
\\
&&
\hspace{1.3cm}
-\left.\frac{1}{n}\tra\{{{\bf C}_n({\bf S}_{x}^{-1}(f){\bf S}_{x,k}^{'}(f))\}}\right|=0.
\end{eqnarray}
Then from \eqref{eq:Def C}, it follows that
\begin{eqnarray}
\nonumber
&&
\lim_{n \rightarrow \infty}\frac{1}{n}
\tra\{
{{\bf C}_n({\bf S}_{x}^{-1}(f){\bf S}_{x,k}^{'}(f))\}}
\\
\nonumber
&&
\hspace{2cm}
=\lim_{n \rightarrow \infty}\frac{1}{n}
\tra\{{\rm Diag}
({\bf S}_{x,n}^{-1} {\bf S}_{x,n,k}^{'})\}
\\
\label{eq:Gray7}
&&
\hspace{2cm}
=\int_{0}^{1}
\tra({\bf S}_{x}^{-1}(f) {\bf S}_{x,k}^{'}(f))df.
\end{eqnarray}
This concludes the proof of \eqref{eq: matrice variance limit 1}.
Following the same derivation steps above, \eqref{eq: matrice variance limit 2} is also proven.
\hfill
\QED
%%%%%%%%%%%%%%%%%%%%%%%%%%%%%%%%%%%%%%%%%%%%%%%%%%
\subsection{Proof of Result \ref{Result banded Toeplitz}}
\label{Proof of Result Result banded Toeplitz}
%%%%%%%%%%%%%%%%%%%%%%%%%%%%%%%%%%%%%%%
%
For the mean-related term, we get from \eqref{eq:mean1}:
\begin{eqnarray}
\nonumber
\boldsymbol{\Sigma}_{{\bf y}_n}
[(\mathds{1}_{n}\otimes {\bf I}_{m})]_{[i,1]}
&=&
\\
&&
\label{eq:banded block1}
\hspace{-4cm}
\left\{
 \begin{array}{c}
\sum_{k=1-i}^q{\bf R}_x(k)={\bf S}_x(0)-\sum_{k=-q}^{-i}{\bf R}_x(k)\\
\hspace{2.3cm} \mbox{for}\ 1 \le  i \le q\\
\sum_{k=-q}^q{\bf R}_x(k)={\bf S}_x(0)\ \mbox{for}\ q+1\le i\le n-q\\
\sum_{k=-q}^{n-i}{\bf R}_x(k)={\bf S}_x(0)-\sum_{k=n-i+1}^{q}{\bf R}_x(k)\\
\hspace{3.6cm} \mbox{for}\ n-q+1\le i\le n.
\end{array}\right.
\end{eqnarray}
This implies that
\begin{equation}
\label{eq:banded block2}
\boldsymbol{\Sigma}_{{\bf y}_n}(\mathds{1}_{n}\otimes {\bf I}_{m})
-(\mathds{1}_{n}\otimes  {\bf I}_{m}){\bf S}_x(0) 
=-{\boldsymbol \Upsilon}_n, 
\end{equation}
where the $mn \times m$ matrix ${\boldsymbol \Upsilon}_n$ gathers the $q(q+1)$  missing terms ${\bf R}_x(k)$.
From \eqref{eq:banded block2}, we obtain \eqref{eq:mean7} again,
which also implies \eqref{eq:mean8}
where here
%$\|\mathds{1}_{n}^T\otimes {\bf I}_{m}\|_2=\sqrt{n}$ and 
$\|{\boldsymbol \Upsilon}_n\|_2\le h(q)M$ (using $\|{\bf A}\|_2^2=\|{\bf A}^T{\bf A}\|_2)$ with $h(q)\pardef \sqrt{q(q+1)(2q+1)/3}$.
Then using $\|\boldsymbol{\Sigma}_{{\bf y}_n}^{-1}\|_2 <c_r$
from the assumptions given in Section \ref{sec:Preliminaries on multidimensional stationary SIRP}, we get
\begin{eqnarray}
\nonumber
\|\frac{1}{n}(\mathds{1}_{n}^T\otimes {\bf I}_{m})
\boldsymbol{\Sigma}_{{\bf y}_n}^{-1}
(\mathds{1}_{n}\otimes {\bf I}_{m})
-{\bf S}_x^{-1}(0)\|_2
&&
\label{eq:banded block5}
\\
&&
\hspace{-3cm}
\le \frac{h(q)c_rM}{\sqrt{n}}  \| {\bf S}_x^{-1}(0)\|_2.
\end{eqnarray}
Multiplying on the left and right side by ${\boldsymbol \mu'}^T_{{\bf x}_k}$ and ${\boldsymbol \mu}'_{{\bf x}_\ell}$, respectively, and then applying the Cauchy-Schwarz inequality, yields~\eqref{eq:Result banded Toeplitz1}.
\hfill
\QED
%%%%%%%%%%%%%%%%%%%%%%%%%%%%%%%%%%%%%%%%%%%
%%%%%%%%%%%%%%%%%%%%%%%%%%%%%%%%%%%%%%%%%

For the covariance-related term, we now analyze the sequence
$\mathrm{tr}(\boldsymbol{\Sigma}_{\mathbf{y}_n}^{-1} \boldsymbol{\Sigma}_{\mathbf{y}_n,k}')$.
The key tool used here is the following lemma proven in the  supplemental material.
\begin{lemma}
\label{decompositionSigma}
The following decomposition holds
\begin{equation}
\label{eq:banded block6}
{\bf W}_{n,m}^H \boldsymbol{\Sigma}_{{\bf y}_n}{\bf W}_{n,m}
=
{\rm Diag}({\bf S}_{x,n})
+\boldsymbol{\Phi}_n,
\end{equation}
where $\boldsymbol{\Phi}_n$ is the $mn \times mn$ block structured matrix given by:
\begin{equation}
\label{eq:banded block7}
\boldsymbol{\Phi}_n
=
{\bf W}_{n,m}^H \boldsymbol{\Delta}_{n}{\bf W}_{n,m},
\end{equation}
where $ \boldsymbol{\Delta}_{n}$ is the  $nm \times nm$ block symmetric-Toeplitz matrix  whose first block row is $({\bf 0},..., {\bf 0},{\bf R}_x^T(q),...,{\bf R}_x^T(1))$.
\end{lemma}
Since ${\boldsymbol \Delta}_{n}$ has non-zero entries only in the first and last $q$ block rows and columns for $n>2q$, it admits the
factorization:
\begin{equation}
\label{eq:rank 1}
{\boldsymbol \Delta}_{n}
={\bf I}_{q,m}{\boldsymbol \Delta}_{n}^r{\bf I}_{q,m}^T,
\end{equation}
where ${\bf I}_{q,m}$ is the $mn\! \times\! 2qm$ selection matrix
$\begin{bmatrix}
{\bf I}_{qm} & {\bf 0}\\
 {\bf 0} &  {\bf 0}\\
 {\bf 0} & {\bf I}_{qm}\\
\end{bmatrix}$
and ${\boldsymbol \Delta}_{n}^r$ is a $2qm \times 2qm$ block symmetric matrix.
Consequently rank$(\boldsymbol{\Phi}_n)=$ rank$(\boldsymbol{\Delta}_n) \le 2qm$.
Because ${\bf W}_{n,m}$ is unitary and the spectral norm is invariant under unitary transformations, it follows that $\|\boldsymbol{\Phi}_n\|_2=\|\boldsymbol{\Delta}_{n}\|_2$.
Moreover, since each row and column of $\boldsymbol{\Delta}_n$ contains at most $q$ non-zero blocks, each bounded by $\|\mathbf{R}_x(k)\|_2 \leq M$, it follows that
$
r(\boldsymbol{\Delta}_n) \pardef \max_i\left[  \sum_{j=1}^n \|[\boldsymbol{\Delta}_n]_{[i,j]}\|_2\right] \leq qM$, and $
c(\boldsymbol{\Delta}_n) \pardef \max_j\left[  \sum_{i=1}^n \|[\boldsymbol{\Delta}_n]_{[i,j]}\|_2\right] \leq qM$.
To this end, we state the following lemma, whose proof is provided in the  supplemental material.
\begin{lemma}
\label{block Schur}
The following inequality holds:
\begin{equation}
\label{eq:block Schur inequality}
\|{\boldsymbol \Delta}_{n}\|_2 \le \sqrt{r({ \boldsymbol \Delta}_{n}) c({ \boldsymbol \Delta}_{n}) }.
\end{equation}
\end{lemma}
Using this inequality together with the bounds
$r(\boldsymbol{\Delta}_n),\,c(\boldsymbol{\Delta}_n)\le qM$,
 we obtain the upper bound:
\begin{equation}
\label{eq:banded reminder1}
\|\boldsymbol{\Phi}_n\|_2
=
\|{\boldsymbol \Delta}_{n}\|_2 \le qM.
\end{equation}
An analogous decomposition holds for the derivative
matrices:
\begin{equation}
\label{eq:banded block8}
{\bf W}_{n,m}^H\! \boldsymbol{\Sigma}_{{\bf y}_{n,k}}^{'}\!{\bf W}_{n,m}
=
{\rm Diag}({\bf S}_{x,n,k}^{'})
+
\boldsymbol{\Phi}_{n,k}^{'},
\end{equation}
with
 $\|\boldsymbol{\Phi}_{n,k}^{'}\|_2 \le qM'_k$ and rank$(\boldsymbol{\Phi}_{n,k}^{'}) \le 2qm$.
Using the unitary invariance of the trace, we have:
\begin{eqnarray}
\nonumber
\frac{1}{n}
\tra(\boldsymbol{\Sigma}_{{\bf y}_n}^{-1}\boldsymbol{\Sigma}_{{\bf y}_n,k}^{'})
&&
\\
\nonumber
&&
\hspace{-2cm}
=\frac{1}{n}\tra[({\bf W}_{n,m}^H \boldsymbol{\Sigma}_{{\bf y}_n}{\bf W}_{n,m})^{-1}({\bf W}_{n,m}^H\! \boldsymbol{\Sigma}_{{\bf y}_{n,k}}^{'}\!{\bf W}_{n,m})]
\\
\label{eq: matrice variance limit}
&&
\hspace{-2cm}
=t_{1,n}+t_{2,n},
\end{eqnarray}
with $t_{1,n}
\pardef \frac{1}{n}
\tra[({\rm Diag}({\bf S}_{x,n})
+\boldsymbol{\Phi}_n)^{-1}{\rm Diag}({\bf S}_{x,n,k}^{'})]$ and
$t_{2,n}
\pardef \frac{1}{n}
\tra[({\rm Diag}({\bf S}_{x,n})
+\boldsymbol{\Phi}_n)^{-1}\boldsymbol{\Phi}_{n,k}^{'})]$.

Let us  first consider $t_{2,n}$.  
Applying the trace inequality
\cite[Th.3.3.14]{Horn1991}
\begin{equation}
\label{eq: trace inequality}
|\tra({\bf A}^T{\bf B})| \le \|{\bf A}\|_2  \|{\bf B}\|_*,
\end{equation}
with $\|({\bf W}_{n,m}^H \boldsymbol{\Sigma}_{{\bf y}_n}{\bf W}_{n,m})^{-1}\|_2=\|\boldsymbol{\Sigma}_{{\bf y}_n}^{-1}\|_2$, 
we obtain:
\begin{equation}
\label{eq:bound T2}
|t_{2,n}|
\le \frac{1}{n}\|\boldsymbol{\Sigma}_{{\bf y}_n}^{-1}\|_2 \|\boldsymbol{\Phi}_{n,k}^{'}\|_*.
\end{equation}
Since $\|\boldsymbol{\Phi}_{n,k}^{'}\|_* \le$ rank$(\boldsymbol{\Phi}_{n,k}^{'})\|\boldsymbol{\Phi}_{n,k}^{'}\|_2 \le 2q^2m M'_k$ and $\|\boldsymbol{\Sigma}_{{\bf y}_n}^{-1}\|_2 < c_r$, it follows that:
\begin{equation}
\label{eq:module T2}
|t_{2,n}|
\le \frac{2q^2m M'_kc_r}{n}.
\end{equation}
Consider now $t_{1,n}$.  Using the resolvant identity:
$\mathbf{A}^{-1} - \mathbf{B}^{-1} = -\mathbf{A}^{-1} (\mathbf{A} - \mathbf{B}) \mathbf{B}^{-1}$ with 
$\mathbf{A} = \mathrm{Diag}(\mathbf{S}_{x,n}) + \boldsymbol{\Phi}_n$ and $\mathbf{B} = \mathrm{Diag}(\mathbf{S}_{x,n})$, we have
\begin{eqnarray}
\nonumber
[{\rm Diag}({\bf S}_{x,n})
+\boldsymbol{\Phi}_n]^{-1}
&-&
{\rm Diag}({\bf S}_{x,n}^{-1})
\\
\label{eq:banded block9}
&&
\hspace{-3cm}
=-[{\rm Diag}({\bf S}_{x,n})
+\boldsymbol{\Phi}_n]^{-1}
\boldsymbol{\Phi}_n
{\rm Diag}({\bf S}_{x,n}^{-1}).
\end{eqnarray}
Substituting into $t_{1,n}$ yields
\begin{equation}
\label{eq:banded block10}
t_{1,n}
=
\frac{1}{n}\tra\{{\rm Diag}({\bf S}_{x,n}^{-1}){\rm Diag}({\bf S}_{x,n,k}^{'})\}
+t_{3,n},
\end{equation}
with
$t_{3,n}\pardef 
-\frac{1}{n}\tra\{
[{\rm Diag}({\bf S}_{x,n})
+\boldsymbol{\Phi}_n]^{-1}
\boldsymbol{\Phi}_n
{\rm Diag}({\bf S}_{x,n}^{-1}){\rm Diag}({\bf S}_{x,n,k}^{'})\}
=
-\frac{1}{n}\tra\{{\rm Diag}({\bf S}_{x,n}^{-1}){\rm Diag}({\bf S}_{x,n,k}^{'})
[{\rm Diag}({\bf S}_{x,n})
\!+\!\boldsymbol{\Phi}_n]^{-1}
\boldsymbol{\Phi}_n\}$.
Applying again \eqref{eq: trace inequality}
\begin{eqnarray}
\nonumber
|t_{3,n}|
&\le&
\frac{1}{n}
\|{\rm Diag}({\bf S}_{x,n}^{-1})\|_2
\|{\rm Diag}({\bf S}_{x,n,k}^{'})\|_2
\\
\label{eq:banded block11}
&&
\hspace{0.8cm}
\|[{\rm Diag}({\bf S}_{x,n})
\!+\!\boldsymbol{\Phi}_n]^{-1}\|_2
\|\boldsymbol{\Phi}_n\|_*
\end{eqnarray}
with
$\|{\rm Diag}({\bf S}_{x,n}^{-1})\|_2\le \frac{1}{c_1}$, 
$\|{\rm Diag}({\bf S}_{x,n,k}^{'})\|_2\le c'_k$,
$\|[{\rm Diag}({\bf S}_{x,n})
\!+\!\boldsymbol{\Phi}_n]^{-1}\|_2 =\|\boldsymbol{\Sigma}_{{\bf y}_n}^{-1}\|_2< c_r$ and
$\|\boldsymbol{\Phi}_{n}\|_* \le$ rank$(\boldsymbol{\Phi}_{n})\|\boldsymbol{\Phi}_{n}\|_* \le 2q^2m M$, gives
\begin{equation}
\label{eq:bound T3}
|t_{3,n}|
\le \frac{2q^2m Mc_rc'_k}{n c_1}.
\end{equation}
Therefore, combining \eqref{eq:bound T2}, \eqref{eq:bound T3} and
${\rm Diag}({\bf S}_{x,n}^{-1}){\rm Diag}({\bf S}_{x,n,k}^{'})={\rm Diag}({\bf S}_{x,n}^{-1}{\bf S}_{x,n,k}^{'})$,
yields
$\frac{1}{n}
\tra(\boldsymbol{\Sigma}_{{\bf y}_n}^{-1}\boldsymbol{\Sigma}_{{\bf y}_n,k}^{'})
-\frac{1}{n}{\rm Diag}({\bf S}_{x,n}^{-1}{\bf S}_{x,n,k}^{'})
=t_{2,n}+t_{3,n}$
which concludes the proof.
\hfill
\QED
%
%%%%%%%%%%%%%%%%%%%%%%%%%%%%%%%%%%%%%%%%%%%%%%%%%%
\subsection{Proof of Result \ref{limit of the three coefficients}}
\label{Proof of  lim a012}
%%%%%%%%%%%%%%%%%%%%%%%%%%%%%%%%%%%%%%%%%ùùù
%%%%%%%%%%%%%%%%%%%%%%%%%%%%%%%%%%%%%
\subsubsection{Proof of rel. \eqref{eq:lim a0}}
%%%%%%%%%%%%%%%%%%%%%%%%%%%%%%%%%%%%%%%%%
%
We establish the result under the additional assumption
$\e(\tau^{-2}) < \infty$. The argument hinges on the
convergence of $z_n \pardef \e(\tau^{-1}|\mathcal{Q}_{mn})$
to $\tau^{-1}$ in $L^2$, which we prove first.

\noindent\textbf{Step 1 ($z_n \stackrel{\rm L^2}{\longrightarrow} \tau^{-1}$).}
Since $\e(\tau^{-2})<\infty$, $\tau^{-1}$ is square-integrable,
and the conditional expectation $\e(\tau^{-1}|\mathcal{Q}_{mn})$
is the orthogonal projection onto the subspace of
$\mathcal{Q}_{mn}$-measurable r.v.. Hence it
minimizes the mean-square error among all such variables; in
particular, for $\frac{mn}{\mathcal{Q}_{mn}}$ (square-integrable
for $mn>4$), 
$\e(\tau^{-1}-z_n)^2
\le
\e\left(\tau^{-1}-\frac{mn}{\mathcal{Q}_{mn}}\right)^2$.
Using the representation $\mathcal{Q}_{mn}=\tau\,\chi^2_{mn}$,
with $\chi^2_{mn}$ independent of $\tau$, the right-hand side
factorizes as
$\e(\tau^{-2})\,\e\left(1-\frac{mn}{\chi^2_{mn}}\right)^2$.
Expanding the square and applying the inverse-moment identity
$\e[(\chi^2_{k})^{-p}]=\frac{\Gamma(\frac{k}{2}-p)}{2^p\,\Gamma(\frac{k}{2})}$,
$k>2p$, with $k=mn$ and $p\in\{1,2\}$ yields
$\e\left(1-\frac{mn}{\chi^2_{mn}}\right)^2
=
1-\frac{2mn}{mn-2}+\frac{(mn)^2}{(mn-2)(mn-4)}
=
\frac{2(mn+4)}{(mn-2)(mn-4)}.$
Since $\e(\tau^{-2})<\infty$, it follows that
$\lim_{n\rightarrow\infty}\e(z_n-\tau^{-1})^2=0$, i.e.,
$z_n \stackrel{\rm L^2}{\longrightarrow} \tau^{-1}$.

\noindent\textbf{Step 2 ($\lim_{n\to\infty} a_{0,mn}=\e(\tau^{-1})$).}
Because $\frac{\mathcal{Q}_{mn}}{mn}z_n$ is
$\mathcal{Q}_{mn}$-measurable, the identity
$\e[X\,\e(Y|\mathcal{Q})]=\e[XY]$ applied with
$X=\frac{\mathcal{Q}_{mn}}{mn}z_n$ and $Y=\tau^{-1}$ gives
\begin{eqnarray}
\nonumber
a_{0,mn}
&\pardef&
\e\left(\frac{\mathcal{Q}_{mn}}{mn}z_n^2\right)
=
\e\left(\frac{\mathcal{Q}_{mn}}{mn}z_n\,
\e(\tau^{-1}|\mathcal{Q}_{mn})\right)
\\
\label{eq:proof a0 1}
&=&
\e\left(\frac{\mathcal{Q}_{mn}}{mn\,\tau}z_n\right)
=
\e\left(\frac{\chi^2_{mn}}{mn}z_n\right).
\end{eqnarray}
Adding and subtracting $\e(z_n)$ and invoking the triangle
inequality, we obtain
\begin{equation}
\label{eq:proof a0 2}
|a_{0,mn}-\e(\tau^{-1})|
\le
|\e((\frac{\chi^2_{mn}}{mn}-1)z_n)|
+|\e(z_n-\tau^{-1})|,
\end{equation}
where, by the Cauchy-Schwarz inequality,
\begin{equation}
\label{eq:proof a0 3}
|\e((\frac{\chi^2_{mn}}{mn}-1)z_n)|
\le
[\e(\frac{\chi^2_{mn}}{mn}-1)^2]^{1/2}
[\e(z_n^2)]^{1/2}.
\end{equation}
Finally, $\e\left(\frac{\chi^2_{mn}}{mn}-1\right)^2=\frac{2}{mn}
\rightarrow 0$, and the $L^2$ convergence established in Step~1
implies $\e(z_n-\tau^{-1})\rightarrow 0$ and
$\e(z_n^2)\rightarrow\e(\tau^{-2})<\infty$. Hence both terms on
the right-hand side of~\eqref{eq:proof a0 2} vanish as
$n\rightarrow\infty$, which yields
$\lim_{n\rightarrow\infty}a_{0,mn}=\e(\tau^{-1})$ and completes
the proof.
\hfill
\QED
%
%%%%%%%%%%%%%%%%%%%%%%%%%%%%%%%%%%%%%%%
%%%%%%%%%%%%%%%%%%%%%%%%%%%%%%%%%%%%%%%
\subsubsection{Proof of Relation~\eqref{eq:lim a2}}
%%%%%%%%%%%%%%%%%%%%%%%%%%%%%%%%%%%%%%%
%
The proof hinges on the almost-sure convergence
$\frac{\mathcal{Q}_{mn}}{mn}=\frac{\chi^2_{mn}}{mn}\,\tau
\stackrel{\rm as}{\longrightarrow}\tau$ (strong law of large
numbers), which forces the posterior of $\tau$ given
$\mathcal{Q}_{mn}=q$ to concentrate, as $n\to\infty$, in a
vanishing neighborhood of $q/(mn)$. Since this concentration
occurs on a single atom for a discrete prior but around a
single point for a continuous prior, the two cases are treated
separately and then combined for the mixed case.

\noindent\textbf{Case 1: discrete distribution ($\beta=1$).}
Assume that $\tau$ is discretely distributed, and let
$t_1 < \dots < t_k < \dots < t_r$ denote its finitely many
positive values, with $P(\tau=t_k)=p_k$ and
$\Delta \pardef \min_{i \neq j}|t_i-t_j| > 0$.
Consider the $r$ disjoint intervals
$I_k \pardef [\,mn\,t_k(1-\epsilon),\, mn\,t_k(1+\epsilon)\,]$,
$k=1,\dots,r$ (with $\epsilon$ to be specified later). The key
quantity is the conditional probability
$P(\tau=t_k | \mathcal{Q}_{mn}=q)$ for $q \in I_k$, which is
controlled by the following lemma, proved in the  supplemental
material.
\begin{lemma}
\label{conditional probability}
For $q \in I_k$, there exist $C_k>0$ and $c>0$, independent of
$n$, such that
\begin{equation}
\label{eq:conditional probability q k}
P(\tau=t_k | \mathcal{Q}_{mn}=q)
\ge
1-C_k\,e^{-cn}.
\end{equation}
\end{lemma}
Since the variance is translation-invariant,
\begin{eqnarray}
\nonumber
\var(\tau^{-1}|\mathcal{Q}_{mn}=q)
\!\!\!\!&=&\!\!\!\!
\var(\tau^{-1}\!-\!t_k^{-1}|\mathcal{Q}_{mn}=q)
\\
\label{eq:var 1}
\!\!\!\!&\le&\!\!\!\!
\e\left[(\tau^{-1}\!-\!t_k^{-1})^2 | \mathcal{Q}_{mn}=q\right]\!\!,
\end{eqnarray}
and, because $\tau$ takes values in $\{t_1,\dots,t_r\}$,
$(\tau^{-1}-t_k^{-1})^2
\le \max_{i\neq j}(t_i^{-1}-t_j^{-1})^2\,
{\bf 1}_{\tau \neq t_k}$.
Consequently, setting
$C' \pardef \max_{i\neq j}(t_i^{-1}-t_j^{-1})^2$,
Lemma~\ref{conditional probability} yields, for $q \in I_k$,
\begin{equation}
\label{eq:var 2}
\var(\tau^{-1}|\mathcal{Q}_{mn}=q)
\le
C'\, P(\tau \neq t_k | \mathcal{Q}_{mn}=q)
\le
C'\, C_k\, e^{-cn}.
\end{equation}
We now decompose the quantity of interest as
\begin{equation}
\label{eq:var 3}
\frac{1}{n}\e\left[\mathcal{Q}_{mn}^2\,
\var(\tau^{-1}|\mathcal{Q}_{mn})\right]
=a'_n+b'_n,
\end{equation}
with
$a'_n \pardef
\frac{1}{n}\e\left[\mathcal{Q}_{mn}^2\,
\var(\tau^{-1}|\mathcal{Q}_{mn})\,
{\bf 1}_{\cup_{k=1}^r I_k}(\mathcal{Q}_{mn})\right]$
and
$b'_n \pardef
\frac{1}{n}\e\left[\mathcal{Q}_{mn}^2\,
\var(\tau^{-1}|\mathcal{Q}_{mn})\,
{\bf 1}_{\cap_{k=1}^r \bar{I}_k}(\mathcal{Q}_{mn})\right]$,
where $\bar{I}_k$ denotes the complement of $I_k$.

\noindent\textit{Bound on $a'_n$.}
Applying \eqref{eq:var 2} on $\cup_k I_k$ gives
\begin{equation}
\label{eq:var 4}
a'_n
\le \frac{1}{n}\e\left(\mathcal{Q}_{mn}^2\right)
C'\, C_*\,e^{-cn},
\end{equation}
where $C_*\pardef \max_k C_k$. Using
$\e\left(\mathcal{Q}_{mn}^2\right)
=\e(\tau^2)\,\e\left[(\chi^2_{mn})^2\right]
=mn(mn+2)\,\e(\tau^2)$,
we obtain
\begin{equation}
\label{eq:var 5}
a'_n
\le m(mn+2)\,\e(\tau^2)\,C'\,C_*\,e^{-cn}.
\end{equation}

\noindent\textit{Bound on $b'_n$.}
We note that the event
$\{\mathcal{Q}_{mn} \in \cap_{k=1}^r \bar{I}_k\}$ implies
$\left|\frac{\chi^2_{mn}}{mn}-1\right| \geq \epsilon$.
Moreover, because $\tau^{-1} \leq t_1^{-1}$ is bounded, the
conditional variance satisfies
$\var(\tau^{-1}|\mathcal{Q}_{mn}=q)\le t_1^{-2}$ for all $q$;
setting $c'\pardef t_1^{-2}$, we get
\begin{equation}
\label{eq:var 6}
b'_n
\le
\frac{1}{n}\,c'\,
\e[\mathcal{Q}_{mn}^2\,
{\bf 1}_{|\frac{\chi^2_{mn}}{mn}-1| \geq \epsilon}].
\end{equation}
Applying the Cauchy-Schwarz inequality to \eqref{eq:var 6}
%   [Fix 9: CS applied to (var 6), not (var 5)]
and invoking the standard Chernoff bound for the $\chi^2_{mn}$
distribution \cite[Th.~2.8]{Boucheron2013}, namely
\begin{equation}
\label{eq:Chernoff bound}
P[|\frac{\chi^2_{mn}}{mn}-1| \geq \epsilon]
\le 2e^{-mn\,c(\epsilon)},
\end{equation}
we deduce
\begin{eqnarray}
\nonumber
b'_n
&\le&
\frac{c'}{n}\sqrt{\e(\mathcal{Q}_{mn}^4)}\;
P[|\frac{\chi^2_{mn}}{mn}-1|
\geq \epsilon]^{1/2}
\\
\label{eq:var 7}
&=&
\sqrt{\e(\tau^4)}O(n)
e^{-\frac{mn}{2}c(\epsilon)}.
\end{eqnarray}
Combining \eqref{eq:var 3}, \eqref{eq:var 4}, and
\eqref{eq:var 7}, both $a'_n$ and $b'_n$ decay exponentially;
hence
$\lim_{n \rightarrow \infty}\frac{1}{n}
\e[\mathcal{Q}_{mn}^2\,
\var(\tau^{-1}|\mathcal{Q}_{mn})]=0$.
Then, from \eqref{eq:a2class} and
Lemma~\ref{property xi 2} together with the identity
$\e[\mathcal{Q}_{mn}^2\,
\var(\tau^{-1}|\mathcal{Q}_{mn})]
=\e[\var(\chi^2_{mn}|\mathcal{Q}_{mn})]$,
we conclude that
$\lim_{n\rightarrow \infty} n\,a_{2,mn}^{\rm Clas}=0$.

\noindent\textbf{Case 2: continuous distribution ($\beta=0$).}
We now assume that $\tau$ is continuously distributed with a
continuous p.d.f. $p_\tau(t)$. Since
$\e(\chi^2_{mn}/mn)=1$, the variance of the conditional
expectation can be written as
\begin{equation}
\label{eq:var 8}
\var[\e(\frac{\chi_{mn}^2}{mn}|\mathcal{Q}_{mn})]
=\e\left\{[\e(\frac{\chi_{mn}^2}{mn}|\mathcal{Q}_{mn})-1]^2 \right\}.
\end{equation}
The following lemma, proved in the  supplemental material,
controls the integrand on compact sets.
\begin{lemma}
\label{conditional}
Let $K \pardef [\alpha_-,\alpha_+]\subset (0,\infty)$.
There exists a constant $C_K>0$ such that, for every $q$
satisfying $\frac{q}{mn} \in K$,
\begin{equation}
\label{eq:conditiona expectation}
|\e(\frac{\chi_{mn}^2}{mn}|\mathcal{Q}_{mn}=q)-1|
\le \frac{C_K}{n}.
\end{equation}
\end{lemma}

We split the right-hand side of \eqref{eq:var 8} as
\begin{equation}
\label{eq:split}
n\e\left\{[\e(\frac{\chi_{mn}^2}{mn}|\mathcal{Q}_{mn})-1]^2 \right\}
=a''_n+b''_n,
\end{equation}
with
$a''_n\pardef
n\e\left\{[\e(\frac{\chi_{mn}^2}{mn}|\mathcal{Q}_{mn})-1]^2
{\bf 1}_{A_n}\right\}$
and
$b''_n \pardef
n\e\left\{[\e(\frac{\chi_{mn}^2}{mn}|\mathcal{Q}_{mn})-1]^2
{\bf 1}_{\bar{A}_n}\right\}$,
where $A_n \pardef \left\{ \alpha_- \le \frac{\mathcal{Q}_{mn}}{mn} \le \alpha_+ \right\}$.
Applying Lemma~\ref{conditional} yields
\begin{equation}
\label{eq:var 9}
a''_n \le \frac{C_K^2}{n}.
\end{equation}
For the tail part, the $L^2$-contraction property of
conditional expectation (conditional Jensen inequality) applied
to $\frac{\chi_{mn}^2}{mn}$ gives
%   [Fix 12: tool named explicitly]
\begin{equation}
\label{eq:var 10}
b''_n
\le n\e[(\frac{\chi_{mn}^2}{mn}-1)^2 {\bf 1}_{\bar{A}_n}].
\end{equation}
For every $0<\epsilon <1$, define
$B_{\epsilon}\pardef \left\{ \frac{\alpha_-}{1-\epsilon} \le \tau \le  \frac{\alpha_+}{1+\epsilon} \right\}$
and
$C_{n,\epsilon}\pardef \left\{ \left|\frac{\chi^2_{mn}}{mn}-1\right| < \epsilon \right\}$.
Since $(B_{\epsilon}\cap C_{n,\epsilon})\subset A_n$, we have
$\bar{A}_n \subset \bar{B}_{\epsilon}\cup \bar{C}_{n,\epsilon}$,
and splitting the right-hand side of \eqref{eq:var 10} accordingly yields
\begin{equation}
\label{eq:var 11}
b''_{n} \le
b''_{1,n,\epsilon} +b''_{2,n,\epsilon},
\end{equation}
with
$b''_{1,n,\epsilon}
\pardef n\e[(\frac{\chi_{mn}^2}{mn}-1)^2 {\bf 1}_{\bar{B}_{\epsilon}}]$
and
$b''_{2,n,\epsilon}
\pardef n\e[(\frac{\chi_{mn}^2}{mn}-1)^2 {\bf 1}_{\bar{C}_{n,\epsilon}}]$.
By the independence of $\tau$ and $\chi_{mn}^2$ and
$\e[(\frac{\chi^2_{mn}}{mn}-1)^2]
=\frac{\var(\chi^2_{mn})}{(mn)^2}=\frac{2}{mn}$,
\begin{eqnarray}
\label{eq:var 12}
b''_{1,n,\epsilon}
&=&
\frac{2}{m}\,P(\tau \notin [ \frac{\alpha_-}{1-\epsilon}, \frac{\alpha_+}{1+\epsilon}]),
\\
\nonumber
b''_{2,n,\epsilon}
&\le&
n\sqrt{\e[(\frac{\chi^2_{mn}}{mn}-1)^4]}\,
\sqrt{P(|\frac{\chi^2_{mn}}{mn}-1| \geq \epsilon )}
\\
\label{eq:var 13}
&=&
O(1)e^{-\frac{mn}{2}c(\epsilon)},
\end{eqnarray}
where the last line uses the Cauchy-Schwarz inequality,
$\e[(\frac{\chi^2_{mn}}{mn}-1)^4]=O(n^{-2})$,
and the Chernoff bound \eqref{eq:Chernoff bound}.

We now take the limits in the following order.
First, choosing $\alpha_-$, $\alpha_+$, and $\epsilon$ makes
$b''_{1,n,\epsilon}$ arbitrarily small, \emph{uniformly in $n$},
since $P(\tau \notin [\frac{\alpha_-}{1-\epsilon}, \frac{\alpha_+}{1+\epsilon}])
\to 0$ as $\alpha_-\to 0$, $\alpha_+\to\infty$, $\epsilon\to 0$.
Then, for these fixed values, $a''_n \to 0$ by
\eqref{eq:var 9} and $b''_{2,n,\epsilon}\to 0$ by
\eqref{eq:var 13} as $n\to\infty$. Hence
$\lim_{n \rightarrow \infty}(a''_n+b''_n)=0$, and from
\eqref{eq:var 8} and \eqref{eq:split},
$\lim_{n \rightarrow \infty}\frac{1}{n}
\var[\e(\chi_{mn}^2|\mathcal{Q}_{mn})]=0$.
Following the proof of Lemma~\ref{property xi 2}, we deduce that
$\lim_{n\rightarrow \infty}n\,a_{2,mn}^{\rm Clas}=-\frac{1}{2m}$.

\noindent\textbf{Case 3: mixed distribution ($0<\beta<1$).}
Finally, suppose that $\tau$ has a mixed distribution. 
By applying the law of total expectation to the r.v.
$\mathcal{Q}_{mn}^2\,\var(\tau^{-1}|\mathcal{Q}_{mn})$ yields
%$P_\tau = \beta P_d + (1-\beta) P_c$, with $P_d$ discrete and
%$P_c$ absolutely continuous. Let $E\in\{d,c\}$ be the latent
%component indicator, with $P(E=d)=\beta$, so that
%$\tau|\{E=d\}\sim P_d$ and $\tau|\{E=c\}\sim P_c$.
%Applying the law of total expectation to
%$\mathcal{Q}_{mn}^2\,\var(\tau^{-1}|\mathcal{Q}_{mn})$
%yields 
rel.~\eqref{eq:lim a2}.
%
%$\frac{1}{n}\e\left[\mathcal{Q}_{mn}^2\,
%\var(\tau^{-1}|\mathcal{Q}_{mn})\right]
%=
%\frac{\beta}{n}\e\left[\mathcal{Q}_{mn}^2\,
%\var(\tau^{-1}|\mathcal{Q}_{mn})\;|dle|\; E=d\right]
%+
%\frac{1-\beta}{n}\e\left[\mathcal{Q}_{mn}^2\,
%\var(\tau^{-1}|\mathcal{Q}_{mn})\;|dle|\; E=c\right].
%$
%%\end{equation}
%%
%Conditionally on $E=d$ (resp.\ $E=c$), the pair
%$(\tau,\mathcal{Q}_{mn})$ follows the model of Case~1
%(resp.\ Case~2), so the exponential concentration estimates of
%Lemmas~\ref{conditional probability} and~\ref{conditional}
%apply to the corresponding component of the posterior
%distribution, and the cross term generated by the law of total
%variance is asymptotically negligible. Combining the two
%conditional limits through \eqref{eq:a2class} and
%Lemma~\ref{property xi 2} therefore yields
%rel.~\eqref{eq:lim a2}.
\hfill
\QED
%%%%%%%%%%%%%%%%%%%%%%%%%%%%%%%%%%%%%%%%%%%%%%%%%%%%%%%%%%%%%
%%%%%%%%%%%%%%%%%%%%%%%%%%%%%%%%%%%%%%%%%%%%%
%%%%%%%%%%%%%%%%%%%%%%%%%%%%%%%%%%%%%%%%%%
\section*{Acknowledgement}
Habti Abeida would like to acknowledge the Deanship of
Graduate Studies and Scientific Research, Taif University for
supporting this work.
%

%%%%%%%%%%%%%%%%%%%%%%%%%%%%%%%%%%%%%%%%%%%%%

%%%%%%%%%%%%%%%%%%%%%%%%%%%%%%%%%%
%%%%%%%%%%%%%%%%%%%%%%%%%%%%%%%%%%%%
\includepdf[pages=-]{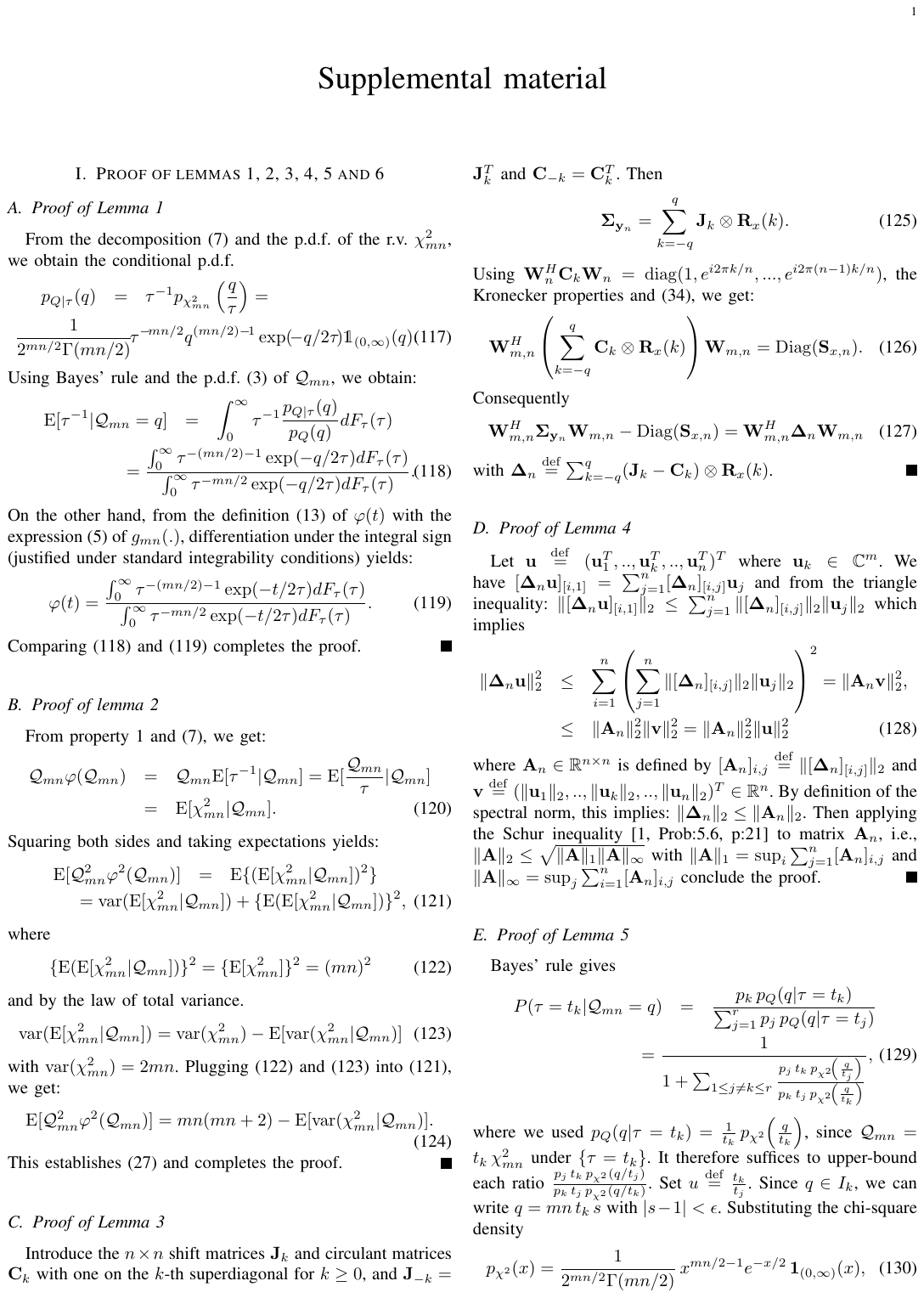}
%%%%%%%%%%%%%%%%%%%%%%%%%%%%%%%%%%%%%%%%%%%%%%%%%%%%ù

\begin{thebibliography}{99}

\bibitem{Yao2003}
K. Yao,
"Spherically invariant random processes: Theory and applications,"
in Communications, Information and Network Security, chap. 6, pp. 313-322, 
Springer Science Business Media, New York, 2003.

\bibitem{Slepian1954}
D. Slepian,
``Estimation of signal parameters in the presence of noise,"
{\it  Trans. IRE Prof. Group Inform. Theory PG IT-3}, pp. 68-89, 1954.

\bibitem{Bangs1971}
W. J. Bangs, ``Array processing with generalized beamformers,''
Ph.D. dissertation, Yale Univ., New Haven, CT, USA, 1971.

\bibitem{Stoica1997}
P. Stoica and R. Moses,
{\it Introduction to Spectral Analysis},
Upper Saddle River, NJ: Prentice-Hall, 1997.

\bibitem{Delmas2004}
J.-P. Delmas and H. Abeida,
``Stochastic Cram\'er-Rao bound for non-circular signals with application to DOA estimation,''
{\it IEEE Trans. Signal Process.},
vol. 52, no. 11, pp. 3192-3199, Nov. 2004.

\bibitem{Besson2013}
O. Besson and Y. I. Abramovich,
``On the Fisher information matrix for multivariate elliptically contoured distributions,''
{\it IEEE Signal Process. Lett.},
vol. 20, no. 11, pp. 1130-1133, Nov. 2013.

\bibitem{Greco2013}
M. Greco and F. Gini,
``Cram\'{e}r-Rao lower bounds on covariance matrix estimation for complex elliptically symmetric distributions,''
{\it IEEE Trans. Signal Process.},
vol. 61, no. 24, pp. 6401-6409, Dec. 2013.

\bibitem{Abeida2019}
H. Abeida and J.-P. Delmas,
``Slepian-Bangs formula and Cram\'er Rao bound for circular and non-circular complex elliptical symmetric distributions,''
{\it IEEE Signal Process. Letters}, vol. 26, no. 10, pp. 1561-1565, Oct. 2019.

\bibitem{Fortunati2019}
S. Fortunati, F. Gini, M. S. Greco, A. M. Zoubir, and M. Rangaswamy,
``Semiparametric CRB and Slepian-Bangs formulas for complex elliptically symmetric distributions,''
{\it IEEE Trans. Signal Process.}, vol. 67, no. 20, pp.  5352-5364, Oct.  2019.

\bibitem{Fortunati2026}
S. Fortunati, J.-P. Delmas, and E. Ollila,
"Nuisance parameters and elliptically symmetric distributions: a geometric approach to parametric and semiparametric efficiency," 
{\it IEEE Trans. Info. Theory}, vol. 72, no. 9, Sept 2026.

\bibitem{Abeida2023}
H. Abeida and J.-P. Delmas,
"Slepian-Bangs formulas for parameterized density generator of elliptically symmetric distributions,"
{\it Signal Processing}, vol. 205, Jan. 2023.

\bibitem{Whittle1953}
P. Whittle,
``The analysis of multiple stationary time series,"
{\it  Journal of the Royal Statistical Society},
vol. 15, no. 1, pp. 125-139, 1953.

\bibitem{Box1970}
G. E. P. Box and G. M. Jenkins,
{\it Times Series Analysis, Forecasting and Control},
San Francisco, CA: Holden-Day, 1970.

\bibitem{Zeira1990}
A. Zeira and A. Nehorai,
``Frequency domain Cramer Rao bound for Gaussian processes,''
{\it IEEE Trans. ASSP.}, vol. 38, no. 6, pp. 1063-1066, June 1990.

\bibitem{Porat1995}
B. Porat,
``On the Fisher Information for the mean of a Gaussian process,''
{\it IEEE Trans. ASSP.}, vol. 43, no. 8, pp. 2033-2035, Aug. 1995.

\bibitem{Porat1993}
B. Porat,
{\it Digital Processing of random variables},
Hoboken, NJ, USA:
Prentice-Hall (Inc.), 1993.

\bibitem{Gray2006}
R. M. Gray,
{\it Toeplitz and Circulant Matrices: A Review}, The essence of knowledge,
Found. Trends Commun. Inf. Theory, vol. 2, no. 3, pp. 155-239, 2006.

\bibitem{Grenander1958}
U. Grenander and G. Szeg\"o,
{\it Toeplitz forms and their applications},
Chelsea Publishing Compagny, New York, 1958.

\bibitem{Delmas2024}
J.-P. Delmas and H. Abeida,
"Generalization of Whittle's formula to compound-Gaussian processes,"
 {\it IEEE Signal Processing Letters}, vol. 31, pp. 746-750, 2024.
 
 \bibitem{Radaelli2023}
M. Radaelli, G. T. Landi, K. Modi, and F. C Binder,
"Fisher information of correlated stochastic processes,"
{\it New Journal of Physics}, no. 25, June 2023.

\bibitem{Gutierrez2008}
J. Gutierrez-Gutierrez and P. M. Crespo,
"Asymptotically equivalent sequences of matrices and Hermitian block Toeplitz matrices with continuous symbols: Applications to MIMO systems,"
{\it IEEE Trans. Info. Theory}, vol. 54, no. 12, pp. 5671-5680, Dec 2008.

\bibitem{Gutierrez2011}
J. Gutierrez-Gutierrez and P. M. Crespo,
{\it Block Toeplitz Matrices:Asymptotic Results and Applications}, The essence of knowledge,
Found. Trends Commun. Inf. Theory, vol. 8, no. 3, pp. 179-257, 2012.

\bibitem{Gutierrez2019}
J. Gutierrez-Gutierrez,
{\it A modified version of the Pisarenko method to estimate the power spectral density of any asymptotically wide sense stationary vector process}, 
Applied Mathematics and Computation, vol. 362, 2019.

\bibitem{Wise1978}
G.L. Wise and N.B. Gallagher,
"On the spherically invariant random processes,"
{\it IEEE Trans. Info. Theory}
vol. 24, no. 1, pp. 118-120, Jan. 1978.

\bibitem{Yao1973}
K. Yao,
"A representation theorem and its applications to spherically invariant random processes,"
{\it IEEE Trans. Info. Theory}, vol. 19, no. 5, pp 600-608, Sept. 1973.

\bibitem{Andrews1974}
D.F. Andrews and C.I. Mallows,
"Scale mixtures of normal distributions,"
{\it Journal Royal Stat Society B}, vol. 36, no.1, pp. 99-102, 1974.

\bibitem{Bickel1993}
P.J. Bickel, C.A.J Klaassen, Y. Ritov, and J.A. Wellner, 
{\it Efficient and Adaptive Estimation for Semiparametric Models}
Johns Hopkins University Press, 1993.

\bibitem{Therrien1992}
C.W. Therrien,
{\it Discrete random signals ans statistical signal processing},
Englewood Cliffs, Prentice Hall, 1992.

\bibitem{Brockwell1990}
P.J. Brockwell and R.A. Davis,
{\it Times series: Theory and Methods},
Springer Verlag, 1990.

\bibitem{Delmas2026}
J.-P. Delmas and H. Abeida,
"The Gaussian data assumption does not always lead to the largest Cram\'er Rao bound,"
{\em IEEE Signal Processing Magazine}, vol. 43 no. 5, pp. 42-46, October 2026.

\bibitem{Horn1991}
R.A. Horn and C.R. Johnson,
{\it Topics in Matrix Analysis}
Cambridge University Press, 1991.

\bibitem{Boucheron2013}
S. Boucheron, G. Lugosi, and P. Massart,
{\it Concentration inequalities: a nonasymptotic theory of independence},
Oxford University Press, 2013.








%\bibitem{Gut2004}
%A. Gut,
%{\it Probability: a graduate course},
%Springer texts in Statistics, 2004.




%\bibitem{Olver1974}
%F.W.J. Olver,
%{\it Asymptotic and special functions},
%Academic Press, 1974.


%
\end{thebibliography}
\end{document}